\documentclass[12pt,reqno,tbtags]{amsart}
\usepackage{amsmath,amsfonts,amssymb,enumerate}
\usepackage[colorlinks=true, pdfstartview=FitV, linkcolor=blue,
 citecolor=blue, urlcolor=red]{hyperref}

\newtheorem{maintheorem}{Theorem}

\newtheorem{conjecture}{Conjecture}
\newtheorem{theorem}{Theorem}[section] \newtheorem{theorema}{Theorem}

\newtheorem*{theorem*}{Theorem} \newtheorem{lemma}[theorem]{Lemma}
\newtheorem{proposition}[theorem]{Proposition}
\newtheorem{corollary}[theorem]{Corollary}

\theoremstyle{definition}{

   }

\theoremstyle{remark}{

  \newtheorem{remark}{Remark}[section] \newtheorem*{remark*}{Remark}
   }

\newenvironment{enumeratei}{\begin{enumerate}[\upshape (i)]}
  {\end{enumerate}}

\newcommand{\ep}{\epsilon}
\newcommand{\even}{{\rm even}}
\newcommand{\odd}{{\rm odd}}
 
\newcommand{\E}{\mathbb E}
\renewcommand{\P}{{\mathbb P}}

\newcommand{\V}{{\mathbb V}}

\newcommand{\one}{{1\!\!1}} 
\renewcommand{\i}{{\rm i}}

\newcommand{\Bigoneb}[1]{\one\Bigl\{ #1 \Bigr\}}
\newcommand{\B}{{\rm B}}
\renewcommand{\L}{{\mathbb L}}
\newcommand{\rmi}{\mathrm{i}} 
\newcommand{\st}{\,:\,}
 
\newcommand{\C}{{\mathbb C}}
\newcommand{\R}{{\mathbb R}}
\renewcommand{\Re}{{\rm Re}} 
\renewcommand{\Im}{{\rm Im}} 
 
\newcommand{\Z}{{\mathbb Z}} 

\newcommand{\deq}{:=} 

\renewcommand{\epsilon}{\varepsilon}

\begin{document}
\title[Counting Partial Hadamard Matrices]{A Fourier-analytic Approach
  to Counting Partial Hadamard Matrices}

\date{\today}

\author{Warwick de Launey} \address{Center for Communication Research,
  4320 Westerra Court, San Diego, California 92121-1969}
\email{warwickdelauney@earthlink.net}

\author{David A. Levin}
\address{Department of Mathematics, University
 of Oregon, Eugene, Oregon 97402-1222}
\thanks{D.A.L. was partially supported by a Summer Research Award from the University of Oregon}
\email{dlevin@uoregon.edu}
\urladdr{http://www.uoregon.edu/\~{}dlevin}



\begin{abstract}
  In this paper, we study a family of lattice walks which are related
  to the Hadamard conjecture.  There is a bijection between
  paths of these walks which originate and terminate at the origin and
  equivalence classes of partial Hadamard matrices.    Therefore,  the
  existence of partial Hadamard matrices can be proved by showing that
  there is positive probability of a random walk returning to the
  origin after a specified number of steps.
  Moreover, the number of these designs can be approximated by estimating
  the return probabilities.    We use the inversion formula for
  the Fourier transform of the random walk to provide such estimates.
  We also include here an upper bound, derived by elementary methods,
  on the number of partial Hadamard.
\end{abstract}
\maketitle
\section{Introduction}
In this paper, we introduce a family of non-symmetric lattice random
walks with importance to combinatorial design theory. 
Paths of these walks starting and ending at the origin
correspond to partial Hadamard matrices (see below for the
definition).  These walks provide a tool for counting the
number of partial Hadamard matrices, without recourse to the usual
constructive methods adopted in design theory.

For non-negative integers $n$ and $t$, a \emph{partial Hadamard
matrix} is an $n \times t$ matrix with $\pm 1$ entries such that the
inner product between any two distinct rows equals zero.  Note that,
since the rows of a partial Hadamard matrix $D$ form a set of $n$
independent $t$-dimensional real vectors, we must have $t\geq n$.
Notice also that if we negate all the entries in a column of $D$, then
the resulting matrix is also a partial Hadamard matrix.  We say the
two matrices are {\em column-negation equivalent}.  Column-negation
equivalence divides the set of $n\times t$ partial Hadamard matrices
into equivalence classes of cardinality $2^t$.

We now define our walk, and show that each distinct column-negation
equivalence class of $n\times t$ partial Hadamard matrices corresponds to a
distinct walk of length $t$ terminating at $0$.  For an integer 
$n \geq 2$, set $d \deq \binom{n}{2}$, set $\V_m \deq \{-1,1\}^m$, and
let $Z: \V_n \to \V_d$ be defined by
\begin{equation} \label{Eq:ZDefn}
  Z(y) = (y_1y_2, y_1y_3, \ldots, y_{n-1} y_n) \,,
\end{equation}
so that the components of $Z(y)$ enumerate all pairwise products
between the components of $y$.  If 
$Y = [y^{(1)} \cdots y^{(t)}]$ is an $n \times t$ matrix with $t$
column vectors $y^{(1)}, \ldots, y^{(t)}$ belonging to 
$\V_n$, then
\begin{equation*}
  Z(y^{(1)}) + \cdots + Z(y^{(t)}) = 0 
\end{equation*}
if and only if the inner product between any two rows of $Y$ is zero.

Let
\begin{equation*}
  M \deq \{ Z(y) \st  y \in \{-1,1\}^n \} \,,
\end{equation*}
then the map $Z:\V_n \rightarrow M$ is two-to-one, since
$Z(-y) = Z(y)$.  Indeed, the column-negation equivalence
class of $n\times t$ partial Hadamard matrices
\begin{equation*}
  [\pm y^{(1)},\pm y^{(2)},\dots,\pm y^{(t)}]
\end{equation*}
maps to the single $M$-sequence
\begin{equation*}
  ({m}^{(1)},
  {m}^{(2)},\dots,
  {m}^{(t)})
  =
  (Z(y^{(1)}),	     
  Z(y^{(2)}),\dots,
  Z(y^{(t)}))
\end{equation*}
of length $t$ such that $\sum_{i}^t m^{(i)}=0$.
Thus, the number of $n\times t$ partial Hadamard matrices is equal to
$2^{nt}$ times the probability that a random walk $(X_t)$ with
increments drawn uniformly from $M$ returns to the origin.

We write $P_n^t(x,y)$ for the $t$-step transition matrix
for $(X_t)$:
\begin{equation*}
  P_n^t(x,y) \deq \P(X_t = y \mid X_0 = x) \,.
\end{equation*}
The random walk $(X_t)$ has a number of unusual features.  It has
dimension  $d=\binom{n}{2}$ but exponentially many (i.e., $2^{n-1}$)
possible increments, each with norm approximately $n/\sqrt{2}$.  
Thus, although, for $n$ fixed, the usual functional central limit theorem
applies (after proper rescaling of space and time, the walk converges in
distribution to Brownian motion), the walk has special discrete
structure which cannot be ignored.  In particular,
\begin{itemize}
  \item The support $\L_d$ of the walk, the smallest subgroup of $\Z^d$
  containing $M$, is a strict subgroup of $\Z^d$,
  \item the walk has period $4$, and
  \item the set of increments is non-symmetric.
\end{itemize}
Moreover, since there is no $n\times t$
partial Hadamard matrix for $t<n$ or $t\not\equiv0\pmod{4}$, we must
have $P_n^t(0,0) = 0$ if either $t < n$ or $t\not\equiv0\pmod{4}$.  It
is conjectured, but not proved, that the converse is true:
\begin{conjecture}
  For $n\geq3$, $P_n^t(0,0) > 0$ if and only if $t\geq n$ and $t
  \equiv 0 \pmod{4}$.
\end{conjecture}
This conjecture is equivalent to the Hadamard Conjecture, which
asserts that there is a Hadamard matrix of order $n$ for all $n$
divisible by four.  The following result is implied by recent work by
de Launey and Gordon \cite{dLG:CHC} and Graham and
Shparlinski~\cite{GS:RSA}. 
\begin{theorema} \label{Thm:A}
  For $\epsilon>0$ and all sufficiently large $n$, 
  if $t - 2t^{\frac{113}{132} + \epsilon} > 2n$ and
  $t \equiv 0 \pmod{4}$, then
  $P_n^t(0,0) > 0$.
\end{theorema}
\begin{proof}
  Towards the end of their paper \cite{GS:RSA}, 
  Graham and Shparlinski note that the
  construction described in \cite{dLG:CHC} applies without recourse to
  the Extended Riemann Hypothesis if one replaces the exponent $7/12$
  in Theorem Theorem 1.2 of \cite{dLG:CHC} with the exponent $113/132$.  Thus
  for all sufficiently large $n$, there is an $n\times t$ partial
  Hadamard matrix whenever $n\leq
  \frac{t}{2}-t^{\frac{113}{132}+\epsilon}$.
\end{proof}
Thus, even before we begin our investigation of the walk $(X_t)$, we
know that there is a non-zero probability that our random walk returns
to its start after about $2n$ steps.  However, the proof depends on
deep number-theoretic results concerning the existence of primes in
short arithmetic sequences and special combinatorial constructions
needed to prove the asymptotic existence of Hadamard matrices
\cite{Crai}.  We hope that analytic techniques along the lines
described in this paper will provide more direct proofs for theorems
like Theorem~\ref{Thm:A}.  Indeed, this paper contains a direct proof
for the following result:
\begin{maintheorem}\label{Thm:P00b}
  Let $\epsilon>0$. For all sufficiently large $n$, if
  $t \geq n^{12+\epsilon}$ and $t \equiv 0 \pmod{4}$, then $P_n^t(0,0)
  > 0$. 
\end{maintheorem}
This result is much weaker than Theorem~\ref{Thm:A}.  However, the
proof of Theorem~\ref{Thm:P00b} offers a number of advantages.
Firstly, it generalizes to give results (which we derive elsewhere)
for other kinds of designs such as balanced incomplete block designs
and difference matrices.

Secondly, our analysis provides an accurate asymptotic formula for the
number of distinct designs -- a result which is not available even in
the special case of partial Hadamard matrices.  Specifically, for $t$
large and $n$ fixed, it is possible to prove a local central limit
theorem for $P_n^t(0,0)$, yielding the following asymptotic formula
for the number of partial Hadamard matrices:

\begin{maintheorem} \label{Thm:AsymFormula} Let $N_{n,t}$ be
  the number of partial Hadamard matrices of dimension $n \times t$,
  and let $d = \binom{n}{2}$.  Then
  \begin{equation} \label{Eq:AsymptoticFormula} 
    N_{n,4t} 
      = [1 + o(1)]
      2^{2d-n+nt +1}(8\pi t)^{-d/2}
      \quad \text{as } t \to \infty \,.
  \end{equation}
\end{maintheorem}

It should be emphasized that to apply the standard local limit theorem
(e.g.\ \cite[P9 on p.\ 79]{spitzer}) to our walk, we must first
transform the walk so that, when sampled at multiples of $4$, it is
strongly aperiodic on $\Z^d$.  However, as indicated above, the
lattice $\L_d$ has a non-trivial structure, leading us to instead
prove Theorem \ref{Thm:AsymFormula} directly from first
principles.  The proof uses the inversion formula (see, for
example, \cite[P3, p.\ 57]{spitzer})
\begin{equation} \label{Eq:Inversion} P_n^t(0,0) = \frac{1}{(2\pi)^{d}}
  \int_{[-\pi,\pi]^d} \psi(\lambda)^t d\lambda \,,
\end{equation}
where, for $\lambda \in \R^d$, the characteristic function
$\psi(\lambda)$ is defined to be the expectation
$\psi(\lambda) \deq 2^{-{n+1}} \sum_{x \in M}e^{\rmi \lambda \cdot x}$.
Following the general approach outlined in \cite{spitzer}, we observe
that the dominant contribution to the integral on the right-hand side
of \eqref{Eq:Inversion} is from the neighborhoods of $\lambda$ with
$|\psi(\lambda)|=1$.  The number and deployment of these neighborhoods
depends on the type of design being investigated.  This paper gives a
complete discussion of these neighborhoods for the walk corresponding
to partial Hadamard matrices.  This direct approach has the benefit of
yielding error estimates in \eqref{Eq:AsymptoticFormula}, and indeed
these are strong enough to prove Theorem~\ref{Thm:P00b}.

Thirdly, we obtain upper bounds for the number of partial Hadamard
matrices.  For example, we will prove the following theorem.
\begin{maintheorem} \label{Thm:BranchingBound}
  There are at most $2^{\binom{n+1}{2}}$ distinct Hadamard matrices of
  order $n$.
\end{maintheorem}
Since there are precisely $2^{n^2}$ distinct $n\times n$
$(-1,1)$-matrices, our result shows that the set of Hadamard matrices
occupies at most about one square root of the entire search space.
While our result is doubtless very weak, it shows that even for small
orders being Hadamard is very rare.  

It is worth pausing to emphasize that, when $t = n$, the integral on
the right-hand side of \eqref{Eq:Inversion} \emph{exactly counts the
number of Hadamard matrices}.  Therefore, a positive resolution to
the Hadamard conjecture is possible if it can be shown that this
integral is positive.  While we have not thus succeeded, we have been
able to approximate the integral to obtain new results on the number
of partial Hadamard matrices, and \emph{we have done so without
constructing a single design}.  Thus the integral on the right-hand
side of \eqref{Eq:Inversion} might lead to a non-constructive proof of
the Hadamard Conjecture.  While we have left open the important (and
probably difficult) problem of obtaining sharper estimates for the
integral in equation \eqref{Eq:Inversion} in the region close to
$t=n$, this paper at the very least introduces an interesting
non-symmetric lattice random walk, where an understanding of the early
(rather than the asymptotic) behavior of the transition probabilities
for the walk is paramount.

\bigskip

The rest of this paper is organized as follows:
In Section \ref{Sec:AofI}, we break up the integral
on the right-hand side of \eqref{Eq:Inversion} into
manageable pieces.  In Section \ref{Sec:EstPsi},
we obtain estimates on the characteristic function
$\psi(\lambda)$.   These estimates are used
in Section \ref{Sec:EstProb} to obtain bounds on the
return probabilities $P_n^t(0,0)$, from which
Theorem \ref{Thm:AsymFormula} is derived.
Theorem \ref{Thm:P00b} is contained in
Theorem \ref{Thm:NAbundExist}.
In Section \ref{Sec:Branching}, we prove 
Theorem~\ref{Thm:BranchingBound}, which is part of
Corollary~\ref{Cor:BranchingBound}.

\section{Anatomy of the Integral} \label{Sec:AofI}
In this section, we divide the region of integration for the integral
\begin{equation}\label{Eq:TheIntegral}
  I(d,t) \deq \int_{[-\pi,\pi]^d}\psi(\lambda)^t d\lambda \,,
\end{equation}
into manageable pieces.


We define the closed boxes
\begin{align*}
  \B_\delta(\lambda) & \deq
  \lambda + [-\delta, \delta]^d
  = \left\{ \mu \in \R^d \st \max_{1 \leq i < j \leq n} | \mu_{\{i,j\}}
    - \lambda_{\{i,j\}} | \leq \delta
     \right\}, \\
  \B_\delta & \deq [-\delta,\delta]^d =  \B_\delta(0) \,.
\end{align*}

Let
\begin{equation*}
  \Lambda \deq \{\lambda\in \B_\pi \st |\psi(\lambda)|=1\} \,.
\end{equation*}
Since $0 \in \Lambda$, 
the set $\Lambda$ is non-empty.  
Since, whenever $|\psi(\lambda)| < 1$,  the magnitude of
$\psi(\lambda)^t$ drops rapidly as $t$ grows, it is natural to suppose
that the bulk of the integral \eqref{Eq:TheIntegral} is accounted for by
points in $\B_\pi$ which are near an element of $\Lambda$.
Consequently, we divide the region $\B_\pi$ of integration into 
the small pieces, $\{\B_\delta(\lambda) \st \lambda \in \Lambda\}$,
and the remaining piece
\begin{equation}\label{Eq:ResidualRegionDefinition}
  R_\delta 
  \deq \B_\pi \setminus\bigcup_{\lambda\in\Lambda} \B_\delta(\lambda) \,.
\end{equation}
We then estimate the integral \eqref{Eq:TheIntegral} by combining
our estimates for each of the pieces. 
The parameter $\delta\in(0,\pi/4)$ determines
the sizes of the regions, and will be adjusted as needed.

\begin{proposition} \label{Prop:RetProbPrimSecond}
  For $\delta \in (0, \pi/4)$,
  \begin{equation}\label{Eq:IntConsol}
    I(n,4t) 
    =2^{2d-n+1}
    \int_{\B_\delta}\psi(\lambda)^{4t} d\lambda
    +
    \int_{R_\delta}\psi(\lambda)^{4t} d\lambda\,,
  \end{equation}
  and, if $t \not\equiv 0 \pmod{4}$, then $I(n,t) = \int_{R_\delta}
  \psi(\lambda)^t d\lambda$.
\end{proposition}

Proposition \ref{Prop:RetProbPrimSecond} will follow from
the following two lemmas.
\begin{lemma} \label{Lem:LambdaFacts}
  Let $\Lambda$ be the set of $\lambda$ with $|\psi(\lambda)|=1$.
  \begin{enumeratei}
    \item \label{Part:Psilambda2}
      If $\lambda \in \Lambda$, then
      $\psi(\lambda + \gamma) = \psi(\lambda)\psi(\gamma)$.
    \item \label{Part:LL0}
      If
      \begin{equation*}
        \Lambda_0 \deq \{ \lambda \in \R^d \st \lambda_{\{i,j\}} \in
        \{0, \pm \pi/2, \pi \} \text{ for all } 1 \leq i < j \leq n
        \}\,,
      \end{equation*}
      then $\Lambda \subset \Lambda_0$.
  \end{enumeratei}
\end{lemma}

\begin{remark} \label{Rmk:LambdaClosed}
  Notice that if $|\psi(\lambda)| = |\psi(\gamma)| = 1$, then
  Lemma~\ref{Lem:LambdaFacts}\eqref{Part:Psilambda2} implies that $|\psi(\lambda + \gamma)| =
  1$.  In other words, $\Lambda$ is closed under addition modulo
  $2\pi$.
\end{remark}

\begin{lemma} \label{Lem:Psilambda1}
  The multi-set
  $\{ \psi(\lambda) \st \lambda \in \Lambda\}$
  contains exactly the elements $\pm 1, \pm i$, each
  appearing $2^{2d - n - 1}$ times.
\end{lemma}

We will prove Lemma \ref{Lem:LambdaFacts} and Lemma
\ref{Lem:Psilambda1} after deriving
Proposition \ref{Prop:RetProbPrimSecond} from them.

\begin{proof}[Proof of Proposition \ref{Prop:RetProbPrimSecond}]
  For $\delta < \pi/4$, by 
  Lemma~\ref{Lem:LambdaFacts}\eqref{Part:LL0}, the boxes  
  $\{\B_\delta(\lambda)\}_{\lambda \in \Lambda}$ 
  are disjoint.  Thus,
  \begin{equation} \label{Eq:IntAdd}
    I(n,t) = \sum_{\lambda \in \Lambda} \int_{\B_\delta(\lambda)}
    \psi(\gamma)^t d\gamma + \int_{R_\delta} \psi(\gamma)^t d\gamma
    \,.
  \end{equation}
  By Lemma~\ref{Lem:LambdaFacts}\eqref{Part:Psilambda2},
  we have
  $\int_{\B_\delta(\lambda)} \psi(\gamma) d\gamma
    = \psi(\lambda) \int_{\B_\delta} \psi(\gamma)d\gamma$,
  which together with \eqref{Eq:IntAdd} shows that
  \begin{equation} \label{Eq:IntAdd2}
    I(n,t) = \sum_{\lambda \in \Lambda} \psi(\lambda) 
      \int_{\B_\delta} \psi(\gamma) d \gamma
      + \int_{R_\delta} \psi(\gamma)^t d\gamma \,.
  \end{equation}
  This identity
  together with Lemma \ref{Lem:Psilambda1}
  yield
  \begin{equation} \label{Eq:IntAdd3}
    \begin{split}
    I(n,t) & = 2^{2d-n-1}\left[1^t + i^t + (-1)^t + (-i)^t \right]
      \int_{\B_\delta} \psi(\gamma)^t  d\gamma\\
      & \quad + \int_{R_\delta} \psi(\gamma)^t d \gamma \,.
    \end{split}
  \end{equation}
  The sum $1^t + i^t + (-1)^t + (-i)^t$ vanishes unless $t \equiv 0 
  \pmod 4$, in which case it equals $4$.  This observation together
  with \eqref{Eq:IntAdd3} finishes the proof.
\end{proof}  

\begin{proof}[Proof of Lemma \ref{Lem:LambdaFacts}]
  We prove Part \eqref{Part:Psilambda2}.
  First, observe that
  \begin{equation} \label{Eq:LambdaConst}
    \lambda \in \Lambda \qquad \Longleftrightarrow \qquad
    e^{\i \lambda\cdot Z(y)} = e^{\i \lambda\cdot Z(w)}
    \quad (\forall \, y,w \in \V_n) \,.
  \end{equation}
  That is, if $\lambda \in \Lambda$, then 
  $\psi(\lambda) = e^{i \lambda \cdot Z(y)}$ for all $y \in \V_n$.
  Consequently,
  \begin{equation*}
    \psi(\lambda + \gamma) =
    2^{-n} \sum_{y \in \V_n} e^{\i Z(y)\cdot (\lambda + \gamma)}
    = 2^{-n} \sum_{y \in \V_n} \psi(\lambda) e^{\i Z(y) \cdot \gamma}
    = \psi(\lambda)\psi(\gamma) \,.
  \end{equation*}

  Next, we prove Part \eqref{Part:LL0}.
  The equations in the right-hand statement of the equivalence
  \eqref{Eq:LambdaConst} are equivalent to the following system of
  modulo $2\pi$ linear equations:
  \begin{equation} \label{Eq:FullLinEqs}
    \sum_{i<j} \lambda_{\{i,j\}}y_iy_j
      \equiv \sum_{i<j}\lambda_{\{i,j\}}w_iw_j \pmod{2\pi}
      \qquad (\forall \, y,w \in \V_n) \,.
  \end{equation}
  Fix $\lambda \in \Lambda$.
  For $y \in \V_n$ and $k \in \{1,2,\ldots,n\}$, define
  \begin{equation*}
    \hat{y}^{(k)}_j
    = 
    \left\{
    \begin{array}{rl}
      y_j & \text{if } j \neq k, \\
      -y_k & \text{if } j = k \,.
    \end{array}
    \right.
  \end{equation*}
  Taking $w = \hat{y}^{(k)}$ in
  \eqref{Eq:FullLinEqs} shows that
  \begin{equation}\label{Eq:SpLinEqns}
    2\sum_{i \st i\not=k}\lambda_{\{i,k\}}y_i \equiv 0 \pmod{2\pi} 
    \qquad (\forall k \in \{1,2,\ldots, k\}) \,.
  \end{equation}
  Since this holds for any choice of $y \in \V_n$, 
  it holds also for $\hat{y}^{(j)}$, whence
  we have the following two instances of \eqref{Eq:SpLinEqns}:
  \begin{align*}
    2\lambda_{\{j,k\}}y_j +
    2\sum_{i \st i\notin\{j,k\}}\lambda_{\{i,k\}}y_i & \equiv 0
    \pmod{2\pi} \,,  \\
    -2\lambda_{\{j,k\}}y_j 
    + 2\sum_{i \st i\notin\{j,k\}}\lambda_{\{i,k\}}y_i  & \equiv 0
    \pmod{2\pi}
    \,.
  \end{align*}
  Therefore, for all $1 \leq j < k \leq n$, it follows that
  $4\lambda_{\{j,k\}}\equiv0\pmod{2\pi}$, that is,
  $\lambda_{\{j,k\}}\in\{0,\pm\pi/2,\pi\}$. 
\end{proof}

From Lemma \ref{Lem:LambdaFacts}\eqref{Part:LL0}, we know that
$\Lambda \subset \Lambda_0$.  In fact, this inclusion is strict.
To prove Lemma \ref{Lem:Psilambda1}, we need to characterize
further the set $\Lambda$.   In view of this, we introduce
the following two sets:
Let
\begin{align*}
  \Lambda_1 & = \{ \lambda \in \R^d \st \lambda_{\{i,j\}} 
    \in \{0, \pi\} \text{ for all }
    1 \leq i < j \leq n \} \,,\\
  \Lambda_2 & = \{ \lambda \in \R^d \st \lambda_{\{i,j\}}
    \in \{0, \pi/2\} \text{ for all } 1 \leq i < j \leq n \} \,.
\end{align*}
These sets have several important properties.
The set $\Lambda_1$ is closed under addition 
modulo $2\pi$, and $\Lambda_2$ is closed under
addition modulo $\pi$. Furthermore, 
$\Lambda_0 =\Lambda_1+\Lambda_2$, meaning that
every $\lambda^{(0)} \in \Lambda_0$ can be written in the
form
\begin{equation}\label{Eq:Lambda0SetDecomp}
  \lambda^{(0)} \equiv \lambda^{(1)} + \lambda^{(2)}
  \pmod{2\pi},
  \quad \text{where } \lambda^{(1)} \in \Lambda_1, \;
  \lambda^{(2)} \in \Lambda_2 \,.
\end{equation}
Moreover, because $\Lambda_0$ contains $4^d$ elements, and
the sets $\Lambda_1$ and $\Lambda_2$ each contain $2^d$
elements, this representation is unique.

Recall, as noted in Remark \ref{Rmk:LambdaClosed}, $\Lambda$
is closed under addition modulo $2\pi$.  Notice that,
since $e^{\pm \pi i} = -1$, the set $\Lambda$ contains
$\Lambda_1$.  Therefore, the element $\lambda^{(0)} = \lambda^{(1)} +
\lambda^{(2)}$ of $\Lambda_0$ is in $\Lambda$ if and only if
$\lambda^{(2)} \in \Lambda$.  Consequently, if we define
$\Lambda_2^\star = \Lambda \cap \Lambda_2$, then
$\Lambda = \Lambda_1 + \Lambda_2^\star$.   
We will now identify the set $\Lambda_2^\star$.
Note that \eqref{Eq:Lambda0SetDecomp} implies that
any $\lambda \in \Lambda$ can be written uniquely as
\begin{equation} \label{Eq:LambdaDecomp}
  \lambda \equiv \lambda^{(1)} + \lambda^\star 
  \pmod{2\pi}, \quad
  \text{where } \lambda^{(1)} \in \Lambda_1, \;
  \lambda^\star \in \Lambda_2^\star \,.
\end{equation}

Each $\Lambda^{(2)} \in \Lambda_2$ has a combinatorial
characterization.
For each element $\lambda^{(2)} \in \Lambda_2$ we define
a weighted graph $G_\lambda$ on the vertices $\{1,2,\ldots,n\}$
by including an edge $\{i,j\}$ if and only if $\lambda_{\{i,j\}} > 0$.
We say a graph is \emph{even-degree} if all of its vertices
have even degree.  We define
\begin{equation*}
  \Lambda_2^\even = 
  \{\lambda \in \Lambda_2 \st G_\lambda \text{ is even-degree} \}\,.
\end{equation*}
We can now provide a useful characterization of the set
$\Lambda^\star$.

\begin{lemma} \label{Lem:LambdaChar}
  We have $\Lambda_2^\star = \Lambda_2^\even$.
\end{lemma}

To prove Lemma \ref{Lem:LambdaChar}, we will need
to know that even-degree graphs are built-up from
triangle graphs.  We denote by $T_{\{a,b,c\}}$
the graph on the vertices $\{1,2,\ldots,n\}$ with
the edges $\{a,b\},\{b,c\},\{c,b\}$.
Let $G$ and $H$ be graphs on the vertex set $\{1,2,\ldots,n\}$.
Then $G \oplus H$ is the graph with vertices $\{1,2,\ldots,n\}$
which contains the edge $\{i,j\}$ if and only if $\{i,j\}$
is an edge of exactly one of $G$ and $H$.

\begin{lemma} \label{Lem:GDecomp}
  The set $\Lambda_2^\even$ consists of the elements
  $\lambda \in \Lambda_2$ such that 
  $G_\lambda=\oplus_{\{a,b,c\}\in\mathcal{T}_{\lambda}}T_{\{a,b,c\}}$
  for some set of triples ${\mathcal T}_\lambda$. 
\end{lemma}
\begin{proof}
  We proceed by induction on the number $E$ of edges.
  The statement is true for $E = 0$.
  There are no non-empty, even-degree graphs with fewer than three edges,
  and the only even-degree graph with three edges is a triangle.
  Thus, the claim is true for $E=3$.
  Suppose now that it holds for all even-degree graphs 
  with strictly fewer than $m > 0$ edges, and suppose that
  $G_\lambda$ has $m$ edges.
  Since $G_\lambda$ has all degrees even, at least two edges, 
  say $\{a,b\}$ and $\{a,c\}$,
  emanate from the same vertex, say $a$.  So the mod 2 sum
  $G_\lambda\oplus T_{\{a,b,c\}}$ has strictly fewer than
  $m$ edges,  and, since $T_{\{a,b,c\}}$ has all even degrees, so does 
  $G_\lambda\oplus T_{\{a,b,c\}}$.  By the induction hypothesis,
  $G_\lambda\oplus T_{\{a,b,c\}}$ can be decomposed in triangles.
  Since $G_\lambda = (G_\lambda\oplus T_{\{a,b,c\}})\oplus
  T_{\{a,b,c\}}$, it follows that $G_\lambda$ also has such
  a decomposition.
\end{proof}

\begin{proof}[Proof of Lemma \ref{Lem:LambdaChar}]
  First, suppose
  $\lambda \in \Lambda_2^\star \deq \Lambda \cap \Lambda_2$.
  Then the equations \eqref{Eq:SpLinEqns} hold for all
  $k \in \{1,2,\ldots, n\}$.   Equation \eqref{Eq:SpLinEqns}
  holds for $k \in \{1,2,\ldots, n\}$ if and only if
  $2 \lambda_{\{i,k\}} = \pi$ for an even number of $i \in
  \{1,2,\ldots,n\} \setminus \{k\}$.  By definition of
  $G_\lambda$, this holds if and only if the degree of
  vertex $k$ in the graph $G_\lambda$ is even.  Thus, $G_\lambda$ is
  an even-degree graph, that is, $\lambda \in \Lambda_2^\even$.
  Therefore, $\Lambda_2^\star \subset \Lambda_2^\even$.

  Next, suppose that $\lambda \in \Lambda_2^\even$.  
  By Lemma~\ref{Lem:GDecomp}, there exists a set of
  triples ${\mathcal T}_\lambda$ such that
  \begin{equation*}
    G_\lambda = \oplus_{\{a,b,c\}\in {\mathcal T}_\lambda}
    T_{\{a,b,c\}} \,.
  \end{equation*}
  Let $\lambda^{(\{a,b,c\})} \in \Lambda_2^\even$ be defined
  as
  \begin{equation*}
    \lambda^{(\{a,b,c\})}_{\{i,j\}}
    \deq
    \begin{cases}
      \tfrac{\pi}{2} &\text{if } \{i,j\}\subset\{a,b,c\}, \\
      0& \text{otherwise} \,.
    \end{cases}
  \end{equation*}
  Then, 
  \begin{equation*}
    \lambda \equiv \sum_{\{a,b,c\} \in {\mathcal T}_\lambda}
      \lambda^{(\{a,b,c\})} \pmod{2\pi} \,.
  \end{equation*}
  Now, for all $y \in \V_n$,
  \begin{equation}\label{Eq:LambdaTri}
    \lambda^{(\{a,b,c\})}\cdot Z(y)
     =\tfrac{\pi}{2}(y_ay_b+y_ay_c+y_by_c)
       \equiv-\tfrac{\pi}{2} \pmod{2\pi}\,.
  \end{equation}
  Thus, for all $y \in \V_n$,
  \begin{equation*}
    \lambda \cdot Z(y) \equiv -\frac{\pi}{2} | {\mathcal
      T}_\lambda | \pmod{2\pi} \,.
  \end{equation*}
  In particular, $e^{\i \lambda \cdot Z(y)} = e^{- \i \frac{\pi}{2}
    | {\mathcal T}_\lambda| }$, and is independent of $y$.  
  Therefore, $\lambda \in \Lambda \cap \Lambda_2 =: \Lambda_2^\star$.  
  We conclude that
  $\Lambda_2^\even \subset \Lambda_2^\star$.
\end{proof}

\begin{lemma} \label{Lem:Lambda2EvenSize}
  The cardinality of $\Lambda_2^\even$ is $2^{\binom{n-1}{2}}$.
\end{lemma}
\begin{proof}
  Each graph with all degrees even on $n$ vertices corresponds to a
  zero-diagonal $n\times n$ symmetric $(0,1)$-matrix all of whose rows
  and columns   have even weight, and each such matrix corresponds to
  a unique zero-diagonal $(n-1)\times(n-1)$ symmetric $(0,1)$-matrix.
  Since there are exactly $2^{\binom{n-1}{2}}$ such matrices, we have
  $|\Lambda_2^\even|= 2^{\binom{n-1}{2}}$.
\end{proof}

\begin{proof}[Proof of Lemma \ref{Lem:Psilambda1}]
  By Lemma \ref{Lem:Lambda2EvenSize}, the size of
  $\Lambda_2^\even$ equals $2^{d-n+1}$.  By
  \eqref{Eq:LambdaDecomp}, $|\Lambda| = 2^{2d-n+1}$.

  Let $\mu = \lambda^{(\{a,b,c\})}$.
  By \eqref{Eq:LambdaTri}, $\psi(\mu) = -\i$.  
  For all $\lambda \in \Lambda$, by 
  Lemma~\ref{Lem:LambdaFacts}\eqref{Part:Psilambda2},
  $\psi(\lambda + \mu) = -\i \psi(\lambda)$.
  Therefore, the multi-set $\{\psi(\lambda) \st \lambda \in \Lambda\}$
  has the composition as stated in the lemma.
\end{proof}

Proposition~\ref{Prop:RetProbPrimSecond}
leaves us with the problem of computing the integral of
$\psi(\gamma)^{4t}$ over two regions: the \emph{primary region}
$\B_\delta$ and the {\em secondary region} 
$R_\delta = \B_\pi \setminus \bigcup_{\lambda \in \Lambda}
\B_\delta(\lambda)$. 
We conclude this section by dividing the secondary region $R_\delta$
into convenient pieces.  

\begin{lemma} \label{Lem:RDecomp}
  If $\Lambda_2^\odd$ is the set of the elements of
  $\Lambda_2$ whose associated graph has at least one odd degree,
  then 
  \begin{equation}\label{Eq:RDeltaDecomp}
    R_\delta
    = \Bigg[\bigcup_{\lambda\in \Lambda_1+\Lambda_2^\even}
    \B_{\frac{\pi}{4}}(\lambda) \setminus \B_{\delta}(\lambda) 
    \Biggr]
    \cup
    \Biggl[ \bigcup_{\lambda\in \Lambda_1+\Lambda_2^\odd}
    \B_{\frac{\pi}{4}}(\lambda) \Biggr]  \,,
  \end{equation}
  where the sets in the union are disjoint.
\end{lemma}
\begin{proof}
  The unit circle can be divided up into four shifted pieces of
  length $\tfrac{\pi}{2}$ centered on the points $1, 
  e^{\rmi \tfrac{\pi}{2}}, e^{-\rmi \tfrac{\pi}{2}}$, and 
  $e^{\rmi \pi}$.   Therefore,
  any $\gamma\in \B_\pi$ may be written uniquely in the form
  \begin{equation*}
    \gamma\equiv\lambda(\gamma)+\delta(\gamma)\pmod{2\pi}\,,
  \end{equation*} 
  where $\delta(\gamma)\in \B_{\frac{\pi}{4}}$ and
  $\lambda(\gamma)\in \Lambda_0$.  Thus
  \begin{equation*}
    \B_\pi = \bigcup_{\lambda\in \Lambda_0}
    \B_{\frac{\pi}{4}}(\lambda)  \,,
  \end{equation*}
  and
  \begin{equation} \label{Eq:RDecomp1}
    R_\delta
    =
    \Biggl[ \bigcup_{\lambda\in \Lambda_0 \setminus \Lambda}
    \B_{\frac{\pi}{4}}(\lambda) \Biggr]
    \cup \Biggl[\bigcup_{\lambda\in\Lambda}
    \B_{\frac{\pi}{4}}(\lambda) \setminus \B_{\delta}(\lambda)
    \Biggr] \,.
  \end{equation}  
  Recall that $\Lambda_0 = \Lambda_1 + \Lambda_2$, and
  $\Lambda_2 \cap \Lambda = \Lambda_2^\even$.   Since
  $\Lambda_2 = \Lambda_2^\even \cup \Lambda_2^\odd$,
  \begin{equation*}
    \Lambda_0 \setminus \Lambda = \Lambda_1 + \Lambda_2^\odd
    \quad \text{and} \quad
    \Lambda = \Lambda_1 + \Lambda_2^\even \,.
  \end{equation*}
  The identity \eqref{Eq:RDeltaDecomp} now follows from
  \eqref{Eq:RDecomp1}.
\end{proof}

In subsequent sections, we derive upper bounds for the integrals
\begin{equation*}
  \int_{\B_\rho(\lambda)} 
  \psi(\gamma)^{4t}
  d\gamma
\end{equation*}
which depend only on whether $\lambda$ is in $\Lambda_2^\even$ or
$\Lambda_2^\odd$.

\section{Estimates for $\psi(\lambda)$} \label{Sec:EstPsi}
In this section, we obtain estimates for the magnitude, the real part,
and the imaginary part of $\psi(\lambda)$.  As a
corollary, we obtain an upper bound for the integral over
the secondary region.

\begin{lemma} \label{Lem:LemmaB}
  The following bounds hold over the entire region $\B_\pi$:
  \begin{align} \label{Eq:MagnPsiUnivUB}
    |\psi(\lambda)|^{2} &\leq
    \tfrac{1}{2}+\tfrac{1}{2}\prod_{i \st i\not=k}^{n-1}\cos(2\lambda_{\{i,k\}})
    \,,\\ \label{Eq:RePsiUniLB}
    \Re(\psi(\lambda)) &\geq1-\tfrac{1}{2}\|
    \lambda\|^2\,.
  \end{align}
  Suppose $\delta>0$, and that $\lambda\in \B_\delta$.  Then
  \begin{align}\label{RePsi}
    \Re(\psi(\lambda))&=e^{-\frac{1}{2}\|\lambda\|^2}
    \bigl(1+\epsilon_1(\lambda)\bigr)\,,\\
    \label{ImPsi}
    \Im(\psi(\lambda)) &= -\sum_{i<j<k} \lambda_{\{i,j\}}
    \lambda_{\{j,k\}} \lambda_{\{k,i\}} +\epsilon_2(\lambda)\,,
  \end{align}
  where
  $|\epsilon_1(\lambda)|<\tfrac{1}{12}(n\delta)^4
  e^{\frac{1}{2}n^2\delta^2}$,
  and $|\epsilon_2(\lambda)|<\tfrac{1}{12}(n\delta)^4$.
\end{lemma}
\begin{proof}
  We first prove \eqref{Eq:MagnPsiUnivUB}.
  Let
  $y\in \V_n$ and $\lambda\in\R^d$. 
  For $i\in\{1,2,\dots,n\}$,
  define
  \begin{eqnarray*}
    p_i(\lambda) &=& (\lambda_{\{1,i\}},\lambda_{\{2,i\}},\dots,
    \lambda_{\{i-1,i\}},\lambda_{\{i,i+1\}}, \dots,\lambda_{\{i,n\}})\,,\\
    Z_i(y) &=& (y_{1},y_{2},\dots,y_{i-1},y_{i+1}, \dots,y_{n})\,,\\
    P_i(\lambda) &=& (\lambda_{\{j,k\}})_{1\leq j<k\leq n,\, j,k\not=i}\,,\\
    Q_i(y) &=& (y_{j}y_{k})_{1\leq j<k\leq n,\, j,k\not=i}\,.
  \end{eqnarray*}
  Note that $p_i(\lambda)$ and $Z_i(y)$ are in $\R^{n-1}$, 
  and $P_i(\lambda)$
  and $Q_i(y)$ are in $\R^{d-n+1}$.  Also, the maps $p_i$ and
  $P_i$ are linear.  Recalling the definition of $Z(y)$ for
  $y \in \V_n$ in \eqref{Eq:ZDefn}, observe that
  for all $i\in\{1,\ldots,n\}$,
  \begin{equation*}
    Z(y) \cdot \lambda = p_i(\lambda) \cdot y_i Z_i(y) + P_i(\lambda)
    \cdot Q_i(y) \,.
  \end{equation*}
  For all $k\in\{1,2,\cdots,n\}$, 
  \begin{eqnarray*}
    \psi(\lambda)
    &=&2^{-n}\sum_{y \in \V_n}
    e^{\rmi  p_k(\lambda)\cdot y_k Z_k(y)}
    e^{\rmi P_k(\lambda)\cdot Q_k( Z_k(y))}\\
    &=&2^{1-n}\sum_{z\in \V_{n-1}}
    \tfrac{1}{2}\left[e^{\rmi  p_k(\lambda)\cdot z}
      +e^{-\rmi p_k(\lambda)\cdot z}\right]
    e^{\rmi P_k(\lambda)\cdot Q_k(z)}\,.
  \end{eqnarray*}
  Therefore,
  \begin{align*}
    |\psi(\lambda)| 
      & \leq \frac{1}{2^{n-1}} \sum_{z\in \V_{n-1}}\tfrac{1}{2}
      \bigl|e^{\rmi p_k(\lambda)\cdot z}+e^{-\rmi p_k(\lambda)\cdot
        z}\bigr| \\
      & = \frac{1}{2^{n-1}} \sum_{z \in \V_{n-1}} \left| \cos( p_k(\lambda) \cdot
        z) \right| \,.
  \end{align*}
  By Jensen's inequality, 
  \begin{equation} \label{Eq:psi1}
    |\psi(\lambda)|^2 \leq
    2^{-(n-1)} \sum_{z \in \V_{n-1}} \cos^2( p_k(\lambda) \cdot
     z) \,.
  \end{equation}
  Since $2\cos^2 \theta = 1 + \cos(2\theta)$,
  \begin{equation} \label{Eq:psi2}
    |\psi(\lambda)|^2 \leq
    \frac{1}{2}\left[1 + 2^{-(n-1)}\sum_{z \in \V_{n-1}} \cos(2 p_k(\lambda)
    \cdot z) \right]\,.
  \end{equation}
  Since $\sin(-\theta) = \sin(\theta)$, and, since
  $\V_{n-1}$ is a symmetric set, it
  follows that $\sum_{z \in \V_{n-1}} \sin(2p_k(\lambda) \cdot z) =0$.
  Thus, since $e^{\i\theta} = \cos \theta + \i\sin \theta$,
  \begin{equation} \label{Eq:psi3}
     2^{-(n-1)} \sum_{z \in \V_{n-1}} e^{\rmi  2 p_k(\lambda) 
       \cdot z}
     = 2^{-(n-1)} \sum_{z \in \V_{n-1}} \cos(2 p_k(\lambda) \cdot z)
    \,.
  \end{equation}
  Combining \eqref{Eq:psi2} and \eqref{Eq:psi3}
  shows that if $\xi$ is a random uniformly distributed element
  of $\V_{n-1}$, then
  \begin{equation} \label{Eq:PsiPenult}
    |\psi(\lambda)|^2 \leq \frac{1}{2}\left[1
      + \E\left( e^{\rmi 2 p_k(\lambda) \cdot \xi} 
    \right) \right]\,.
  \end{equation}
  The coordinates $\xi_1,\ldots,\xi_{n-1}$ of $\xi$ are
  independent unbiased random $\pm 1$'s.  Therefore, 
  because the components of $p_k(\lambda)$ are
  $\{ \lambda_{\{i,k\}} \st i \neq k\}$,
  \begin{multline} \label{Eq:PsiInd}
    \E\left( e^{\rmi 2 p_k(\lambda) \cdot \xi} \right)
    = \prod_{j=1}^{n-1} \E\left( e^{\rmi 2 [p_k(\lambda)]_j \xi_j}
    \right)\\
    = \prod_{i \st i \neq k} \left[ 
      \frac{1}{2}e^{\rmi 2 \lambda_{\{i,k\}}} +
      \frac{1}{2}e^{-\rmi 2 \lambda_{\{i,k\}}} \right]
    = \prod_{i \st i \neq k} \cos{2 \lambda_{\{i,k\}}} \,.
  \end{multline}
  Substituting \eqref{Eq:PsiInd} into \eqref{Eq:PsiPenult} establishes
  \eqref{Eq:MagnPsiUnivUB}.


  Next we deal with the bounds for $\Re(\psi(\lambda))$ and 
  $\Im(\psi(\lambda))$.   We use the following bounds on the remainder
  in the Taylor expansion of the exponential: For
  $a\geq0$ and $b$ real,
  \begin{align}
    \label{Eq:Taylor1}
    \Bigl|e^{-a}-\sum_{s=0}^k\frac{(-a)^s}{s!}\Bigr|
    &\leq\min\Bigl\{\frac{2|a|^k}{k!},\frac{|a|^{k+1}}{(k+1)!}\Bigr\}\,,\\
    \label{Eq:Taylor2}
    \Bigl|e^{\rmi b}-\sum_{s=0}^k\frac{(\rmi b)^s}{s!}\Bigr|
    &\leq\min\Bigl\{\frac{2|b|^k}{k!},\frac{|b|^{k+1}}{(k+1)!}\Bigr\}\,.
  \end{align}
  (Equation \eqref{Eq:Taylor2} can be found as, for example,
  \cite[equation 26.4]{billingsley}; the derivation of
  equation \eqref{Eq:Taylor1} is similar.)
  Equation \eqref{Eq:Taylor1} with $k=2$ implies
  \begin{equation} \label{Eq:eToNormLa}
    \Bigl|e^{-\frac{1}{2}\|\lambda\|^2}
    -(1-\tfrac{1}{2}\|\lambda\|^2)\Bigr|
    \leq 
    \tfrac{1}{6}\|\lambda\|^4 \,;
  \end{equation}
  equation \eqref{Eq:Taylor2} with $k=2$
  implies
  \begin{equation*}
    \Bigl|e^{\rmi \lambda\cdot Z(y)}- \bigl[ 1+\rmi\lambda\cdot
    Z(y)\bigr]\Bigr| \leq \tfrac{1}{2}(\lambda\cdot Z(y))^2
    \,;
  \end{equation*}
  equation \eqref{Eq:Taylor2} with $k=3$ implies
  \begin{multline*}
    \Bigl|e^{\rmi \lambda\cdot Z(y)}- \big[1+\rmi\lambda\cdot Z(y)
    -\tfrac{1}{2}(\lambda\cdot Z(y))^2
    -\tfrac{\rmi}{6}(\lambda\cdot Z(y))^3\bigr]\Bigr| \\
    \leq
    \tfrac{1}{24}(\lambda\cdot Z(y))^4 \,.
  \end{multline*}
  Since $|\Re(z)| \leq |z|$ and $|\Im(z)|\leq|z|$ for all $z \in \C$,
  we have
  \begin{align}
    \Bigl|\Im(e^{\rmi \lambda\cdot Z(y)})- \bigl[\lambda\cdot
    Z(y)-\tfrac{1}{6}(\lambda\cdot Z(y))^3\bigr]\Bigr| &\leq
    \tfrac{1}{24}(\lambda\cdot Z(y))^4
    \,, \label{Eq:Im1}\\
    \Bigl|\Re(e^{\rmi \lambda\cdot Z(y)})-
    \bigl[1-\tfrac{1}{2}(\lambda\cdot Z(y))^2\bigr] \Bigr| &\leq
    \tfrac{1}{24}(\lambda\cdot Z(y))^4
    \,,\\
    \Bigl|\Re(e^{\rmi \lambda\cdot Z(y)})-1\Bigr| &\leq
    \tfrac{1}{2}(\lambda\cdot Z(y))^2 \,.
  \end{align}
  Let $\xi$ be a uniform random element of $\V_n$.
  From \eqref{Eq:Im1},
  \begin{align} 
    \Bigl| \E\left[ \Im(e^{\rmi \lambda \cdot Z(\xi)})\right]
      & - \E \left[ \lambda \cdot Z(\xi)
        + \tfrac{1}{6}(\lambda \cdot Z(\xi))^3 \right] \Bigr|
     \nonumber \\
    & \leq \E \left|\Im(e^{\rmi \lambda \cdot Z(\xi)})
       - \left[ \lambda \cdot Z(\xi)
        + \tfrac{1}{6}(\lambda \cdot Z(\xi))^3 \right] \right| 
    \nonumber \\
    & \leq \E\left[ \tfrac{1}{24} (\lambda \cdot Z(\xi))^4
    \right] 
    \label{Eq:ImPsi1p} 
  \end{align} 
  Since $\Im$ is linear, $\E\left[ \Im(e^{\rmi \lambda \cdot
    Z(\xi)})\right] = \Im\left( \psi(\lambda) \right)$, whence 
    \eqref{Eq:ImPsi1p}
  implies 
  \begin{equation} \label{Eq:ImPsi1}
    \Bigl| \Im\left( \psi(\lambda)\right) - \E\left[ \lambda \cdot Z(\xi)
      \right] 
      - \tfrac{1}{6} \E\left[ (\lambda \cdot Z(\xi))^3 \right] 
      \Bigr| 
    \leq \E\left[ \tfrac{1}{24} (\lambda \cdot Z(\xi))^4 \right]
      \,. 
  \end{equation}
  Similarly, we have
  \begin{gather}
    \Bigl|\Re(\psi(\lambda))-\bigl(1
    -\tfrac{1}{2}\E[(\lambda\cdot Z(\xi))^2]\bigr)\Bigr| \leq
    \tfrac{1}{24}\E[(\lambda\cdot Z(\xi))^4] \,,
    \label{Eq:Psi1}  \\
    \intertext{and}
    \Bigl|\Re(\psi(\lambda))-1\Bigr| \leq
    \tfrac{1}{2}\E[(\lambda\cdot Z(\xi))^2] 
    \,. \label{Eq:Psi1a}
  \end{gather}
  Our goal now is to compute the above expectations. 
  For all non-negative integers $s$,
  \begin{equation*}
    \E[(\lambda\cdot Z(\xi))^s] 
    =
    \sum_{1\leq k_1<\ell_1\leq n}
    \sum_{1\leq k_2<\ell_2\leq n}
    \dots
    \sum_{1\leq k_s<\ell_s\leq n}
    \prod_{j=1}^s\lambda_{k_j\ell_j}
    \E\Bigl[\prod_{j=1}^s\xi_{k_j}\xi_{\ell_j}\Bigr]\,.
  \end{equation*}
  For each multi-set $S =
  \{\{k_1,\ell_1\},\{k_2,\ell_2\},\dots,\{k_s,\ell_s\}\}$,
  let $N_S$ be the network on the vertices
  $\{1,2,\dots,n\}$ with the edge set $S$, where
  repeated elements in $S$ correspond to multiple
  edges between vertices.
  Observe that
  \begin{equation*}
    \E\left[ \prod_{j=1}^s\xi_{k_j}\xi_{\ell_j} \right]
    =
    \begin{cases}
      1 & \text{if all vertices in $N_s$ have even degree}, \\
      0 & \text{otherwise} \,.
    \end{cases}
  \end{equation*}
%
%
  Therefore,
  \begin{equation*}
    \E[(\lambda\cdot Z(\xi))^s]=\sum_{N_S}w(N_S)\prod_{\{k,\ell\}\in
      S}\lambda_{\{k,\ell\}}\,.
  \end{equation*}
  Here $N_S$ ranges over all the networks on the vertices
  $\{1,2,\dots,n\}$ having $s$ edges and all degrees even, and
  $w(N_S)$ is a multinomial coefficient determined by the number of
  times each edge appears in $N_S$.

  For $s=1$, there are no even-degree networks.  Therefore,
  \begin{equation*}
    \E[\lambda\cdot Z(\xi)]=0\,.
  \end{equation*}  
  For $s=2$, the even-degree networks are the two-vertex networks
  $N_S$ with a single repeated edge
  $S=\{\{k_1,\ell_1\},\{k_1,\ell_1\}\}$, and the weights $w(N_S)$ all
  equal $2!/2!=1$.  Thus
  \begin{equation*}
    \E[(\lambda\cdot Z(\xi))^2]=\|\lambda\|^2\,.
  \end{equation*}
  Equation \eqref{Eq:Psi1a} therefore implies that
  \begin{equation*}
    \Bigl|\Re(\psi(\lambda))-1\Bigr|
    \leq 
    \tfrac{1}{2}\|\lambda\|^2 \,,
  \end{equation*}
  from which \eqref{Eq:RePsiUniLB} follows.

  For $s=3$, the even-degree networks are just triangles 
  with the edges
  \begin{equation*}
    \{\{k_1,k_2\}, \{k_2,k_3\},\{k_1,k_3\}\} \,,
  \end{equation*} 
  where $1\leq k_1<k_2<k_3\leq n$, and the
  weights $w(N_S)$ are all $3!/(1!)^3$, as there are three edges and
  each edge appears just once.  Therefore,
  \begin{equation*}
    \E\left[(\lambda\cdot Z(\xi))^3 \right]
    =\sum_{i<j<k}3!\lambda_{\{i,j\}}\lambda_{\{j,k\}}\lambda_{\{i,k\}} \,.
  \end{equation*}
  Thus the inequalities \eqref{Eq:ImPsi1} and \eqref{Eq:Psi1}
  become
  \begin{align}\label{Eq:ImPsi2} 
    \Bigl|\Im(\psi(\lambda)) + 
    \sum_{i<j<k}\lambda_{\{i,j\}}\lambda_{\{j,k\}}\lambda_{\{i,k\}}
    \Bigr| &\leq \tfrac{1}{24}\E[(\lambda\cdot Z(\xi))^4]
    \,,\\
    \Bigl|\Re(\psi(\lambda))-\bigl(1
    -\tfrac{1}{2}\|\lambda\|^2\bigr)\Bigr| &\leq
    \tfrac{1}{24}\E[(\lambda\cdot Z(\xi))^4] \,.
    \label{Eq:RePsi2} 
  \end{align}
  Moreover, applying the triangle inequality to the inequalities
  \eqref{Eq:eToNormLa} and \eqref{Eq:RePsi2} shows that
  \begin{equation}  \label{Eq:RePsi3a}
    \Bigl|\Re(\psi(\lambda))-e^{-\frac{1}{2}\|\lambda\|^2}\Bigr| \leq
    \tfrac{1}{6}\|\lambda\|^4 +\tfrac{1}{24}\E[(\lambda\cdot
    Z(\xi))^4] \,.
  \end{equation}

  Finally, for $s=4$, there are several classes: (a) $4$-cycles, (b) an
  edge repeated four times, (c) two non-adjacent edges repeated twice,
  and (d) two adjacent edges repeated twice.  Thus
  \begin{align*}
    \E\left[(\lambda\cdot Z(\xi))^4\right]
    =&\sum_{i_1<j_2}\lambda_{\{i_1,i_2\}}^4
    +4!/(2!)^2\sum_{\{i_1,i_2\}\atop{\not=\{i_3,i_4\}}}\lambda_{\{i_1,i_2\}}^2
    \lambda_{\{i_3,i_4\}}^2\\
    &+4!\sum_{\substack{i_1,i_2,i_3,i_4\\ \text{distinct}}}
    \lambda_{\{i_1,i_2\}}\lambda_{\{i_2,i_3\}}
    \lambda_{\{i_3,i_4\}}\lambda_{\{i_4,i_1\}}\,.
  \end{align*}
  Therefore, for $\lambda \in \B_\delta$, 
  \begin{align}
    \E\left[(\lambda\cdot Z(\xi))^4\right]
    & \leq
    \tfrac{1}{2}n^2\delta^4+\tfrac{3}{4}n^4\delta^4+n^4\delta^4 \,,
    \nonumber \\
\intertext{and, because we always assume $n \geq 3$,}
    \E\left[(\lambda\cdot Z(\xi))^4\right] & \leq
    (\tfrac{1}{18}+\tfrac{3}{4}+1)(n\delta)^4
     \leq\tfrac{11}{6}(n\delta)^4 \,.
     \label{Eq:ELZ4}
  \end{align}
  Consequently, the inequalities \eqref{Eq:ImPsi2} and 
  \eqref{Eq:RePsi2} imply
  \begin{align}
    \Bigl|\Im(\psi(\lambda)) +
      \sum_{i<j<k}\lambda_{\{i,j\}}\lambda_{\{j,k\}}\lambda_{\{i,k\}}
      \Bigr| 
      & \leq \tfrac{1}{12}(n\delta)^4 \,, \nonumber \\ 
        \Bigl|\Re(\psi(\lambda))-\bigl(1 -\tfrac{1}{2}\|
          \lambda\|^2\bigr)\Bigr| 
      & \leq \tfrac{1}{12}(n\delta)^4 \,. \label{Eq:RePsi3}
  \end{align}
  This last inequality gives us the estimate on $\Im(\psi(\lambda))$
  claimed in \eqref{ImPsi}.   

  We now use \eqref{Eq:RePsi3} to prove the estimate on
  $\Re(\psi(\lambda))$ stated in \eqref{RePsi}.
  Because $n \geq 3$, we have
  $\|\lambda\|^4\leq(\tfrac{1}{2}n^2\delta^2)^2\leq\tfrac{1}{36}(n\delta)^4$.
  Therefore, the equations \eqref{Eq:RePsi3a} and \eqref{Eq:ELZ4} 
  imply
  \begin{align*}
    \Bigl|\Re(\psi(\lambda))-e^{-\frac{1}{2}\|\lambda\|^2}\Bigr| \leq
    \tfrac{1}{216}(n\delta)^4 + \tfrac{11}{144}(n\delta)^4
    <\tfrac{1}{12}(n\delta)^4\,.
  \end{align*}
  Therefore, since
  $e^{-\frac{1}{2}\|\lambda\|^2}\geq e^{-\frac{1}{2}n^2\delta^2}$
  for $\lambda\in \B_\delta$,
  it follows that
  \begin{equation*}
    \Bigl|\frac{\Re(\psi(\lambda))}{e^{-\frac{1}{2}\|\lambda\|^2}}-1\Bigr|
    \leq
    \frac{\tfrac{1}{12}(n\delta)^4}{e^{-\frac{1}{2}\|\lambda\|^2}}\\
    <\frac{\tfrac{1}{12}(n\delta)^4}{e^{-\frac{1}{2}n^2\delta^2}}\,.
  \end{equation*}
  Therefore,
  \begin{equation*}
    \Re(\psi(\lambda))=e^{-\frac{1}{2}\|\lambda\|^2}
    \left[\frac{\Re(\psi(\lambda))}{e^{-\frac{1}{2}\|\lambda\|^2}}\right]
    =e^{-\frac{1}{2}\|\lambda\|^2}(1+\epsilon_1(\lambda))\,,
  \end{equation*}
  where
  \begin{equation*}
    \bigl|\epsilon_1(\lambda)\bigr|
    =
    \left|\frac{\Re(\psi(\lambda))}{e^{-\frac{1}{2}\|\lambda\|^2}}-1\right|
    \leq 
    \tfrac{1}{12}(n\delta)^4e^{\frac{1}{2}n^2\delta^2}\,.
  \end{equation*}
\end{proof}
We now bound the contribution of the secondary region to the integral
$I(d,t)$.
\begin{proposition}\label{Prop:ResIntBnd} 
  \begin{equation*}
    \left| (2\pi)^{-d} \int_{R_\delta} \psi(\lambda)^t d\lambda
      \right| \leq e^{-\frac{11}{24} t\delta^2} \,.
  \end{equation*}
\end{proposition}
\begin{proof}
  By \eqref{Eq:MagnPsiUnivUB},
  for all $k\in\{1,2,\dots,n\}$,
  \begin{equation*}
    |\psi(\lambda)|^2
    \leq\tfrac{1}{2}+\tfrac{1}{2}\prod_{i \st i\not=k}\cos(2\lambda_{\{i,k\}}) \,.
  \end{equation*}
  Let $\gamma\in R_\delta$.  By Lemma \ref{Lem:RDecomp},
  either there is an element $\lambda \in \Lambda_1 + \Lambda_2^\even$
  such that $\gamma \in \B_{\pi/4}(\lambda) \setminus
  \B_\delta(\lambda)$, or there is an element $\lambda \in \Lambda_1
  + \Lambda_2^\odd$ such that $\gamma \in \B_{\pi/4}(\lambda)$.
%
%
  In the first case, $\gamma = \lambda^{(1)} + \lambda^{(2)} + \mu$,
  where $\mu \in \B_{\pi/4}\setminus\B_{\delta}$ and
  $\lambda^{(i)} \in \Lambda_i$ for $i=1,2$.
  Thus, there are $a,b \in \{0,1\}$ such that
  \begin{equation*}
    |\cos(2\gamma_{\{i,j\}})| =
    |\cos(2\pi a + \pi b + 2\mu_{\{i,j\}})|
    = |\cos(2\mu_{\{i,j\}})| \,.
  \end{equation*}
  Whence,
  \begin{equation*}
    |\psi(\gamma)|^2
    \leq\tfrac{1}{2}
    +|\tfrac{1}{2}\min_{\{i,j\}}\bigl\{\cos(2\gamma_{\{i,j\}})\bigr\}|
    \leq\tfrac{1}{2}+\tfrac{1}{2}\cos(2\delta) =\cos^2\delta\,.
  \end{equation*}
  In the second case, there is a choice of
  $k\in\{1,2,\dots,n\}$ such that an odd number of the components
  $\lambda_{\{i,k\}}$ ($i\not=k$) are $\tfrac{\pi}{2}$.  For this
  value of $k$, we have
  \begin{equation*}
    \prod_{i\not=k}\cos(2\gamma_{\{i,k\}}) \leq 0\,.
  \end{equation*}
  Therefore, in this case, we trivially have
  $|\psi(\gamma)|^2 \leq (1/2) \leq \cos^2\delta$, 
  since $\delta \leq \pi/4$. Therefore,
  \begin{equation*}
    |\psi(\gamma)|^2 \leq \cos^2\delta\qquad(\forall\ \gamma\in R_\delta).
  \end{equation*}

  Using the inequalities $\cos x \leq 1 - x^2/2 + x^4/24$ and
  $1-x \leq e^{-x}$ yields

  \begin{equation*}
    \cos \delta
    \leq 1-\frac{\delta^2}{2}+\frac{\delta^4}{24} 
    \leq e^{-\frac{\delta^2}{2}(1-\frac{\delta^2}{12})}
    \leq e^{-\frac{11}{24}\delta^2}\,.
  \end{equation*}
  Therefore, for all $\gamma\in R_\delta$, we have
  $|\psi(\gamma)^t| <e^{-\frac{11}{24}t\delta^2}$, and hence we
  certainly have
  \begin{equation*}
    (2\pi)^{-d}\left| \int_{R_\delta}\psi(\gamma)^t d\gamma \right|
    \leq
    (2\pi)^{-d}\int_{R_\delta}|\psi(\gamma)^t| d\gamma
    <e^{-\frac{11}{24}t\delta^2}\,.
  \end{equation*}
\end{proof}

\section{An Estimate for the Return Probabilities}
\label{Sec:EstProb}
We use our estimates obtained in the previous section for
$\psi(\lambda)$ to obtain upper and lower
bounds for the integral
\begin{equation}
  P_n^{(t)}(0,0)
  = \frac{1}{(2\pi)^{d}}\int_{\B_\pi}\psi(\lambda)^t d\lambda\,.
\end{equation}
Define
\begin{align}
  U(n,t,\delta) 
  &= \bigl[1+\tfrac{1}{9}(n\delta)^6\bigr]^{t/2}
  \bigl[1+\tfrac{1}{12}(n\delta)^4\bigr]^{t}
  \bigl[1-e^{-t\delta^2}\bigr]^{d/2}\,, \label{Eq:Udef1}\\
  L(n,t,\delta)
  & = \left[ 1 + \tfrac{4}{9} t^2 (n\delta)^6 \right]^{-\frac{1}{2}}
  \bigl[1-\tfrac{1}{12}(n\delta)^4\bigr]^{t}
   \left[1-e^{-\frac{1}{2} t\delta^2} \right]^{d/2} \label{Eq:Ldef1}\,.
\end{align}

\begin{align}
  U(n,4t,\delta) &=\bigl[1+\tfrac{1}{9}(n\delta)^6\bigr]^{2t}
  \bigl[1+\tfrac{1}{12}(n\delta)^4\bigr]^{4t}
  \bigl[1-e^{-4t\delta^2}\bigr]^{d/2}\,, \label{Eq:Udef}\\
  L(n,4t,\delta)
  &=\bigl[1+\tfrac{4}{9}(4t)^2(n\delta)^6\bigr]^{-\frac{1}{2}}
  \bigl[1-\tfrac{1}{12}(n\delta)^4\bigr]^{4t}
  \bigl[1-e^{-2t\delta^2}\bigr]^{d/2} \label{Eq:Ldef}\,.
\end{align}
\begin{theorem}\label{Thm:Bounds1}
  Suppose that $n\delta\in(0,1)$, and let $t$ be a positive integer.
  Let $U(n,4t,\delta)$ and $L(n,4t,\delta)$ be as defined
  in \eqref{Eq:Udef} and \eqref{Eq:Ldef}, respectively.
  Then
  \begin{equation} \label{Eq:P00UB}
    P_{n}^{(4t)}(0,0)
    \leq
    2^{2d-n+1}(8\pi t)^{-d/2}U(n,4t,\delta) + e^{-\frac{11}{6}t\delta^2}\,.
  \end{equation}
  Moreover, if $4t(n\delta)^3<1$, then
  \begin{equation}\label{Eq:P00LB}
    P_{n}^{(4t)}(0,0)
    \geq
    2^{2d-n+1}(8\pi t)^{-d/2}L(n,4t,\delta) - e^{-\frac{11}{6}t\delta^2}\,.
  \end{equation}
\end{theorem}
\begin{proof}
  Rearranging equation \eqref{Eq:IntConsol} we have
  \begin{equation*}
    P_n^{(4t)}(0,0)
    -2^{2d-n+1}(2\pi)^{-d}\int_{\B_\delta}\psi(\gamma)^{4t}d\gamma 
    =(2\pi)^{-d}\int_{R_\delta}\psi(\gamma)^{t}d\gamma\,;
  \end{equation*}
  By Proposition~\ref{Prop:ResIntBnd},
  \begin{equation*}
    \Bigl| 
    P_n^{(4t)}(0,0)
    -2^{2d-n+1}(2\pi)^{-d}\int_{\B_\delta}\psi(\gamma)^{4t}d\gamma 
    \Bigr|
    \leq (2\pi)^de^{-\frac{11}{6}t\delta^2}\,.
  \end{equation*}
  Thus, it is sufficient to prove that, for $t$ satisfying 
  the conditions of the theorem,
  \begin{equation} \label{Eq:Verify}
    (8\pi t)^{-\frac{d}{2}}L(n,4t,\delta)
    <
    (2\pi)^{-d}\int_{\B_\delta}\psi(\gamma)^{4t}d\gamma
    <
    (8\pi t)^{-\frac{d}{2}}U(n,4t,\delta)\,.
  \end{equation}
  First note that, since $\psi(-\gamma)$ is the complex conjugate of
  $\psi(\gamma)$,
  \begin{equation} \label{Eq:RealInt}
    \int_{\B_\delta}\psi(\gamma)^{4t}d\gamma
    =
    \int_{\B_\delta}\Re\bigl(\psi(\gamma)^{4t}\bigr) d\gamma\,.
  \end{equation}
  Therefore, we only need to understand the real part of the powers of
  $\psi$.  In this proof, we employ  
  Proposition~\ref{RealPartOfPowerProposition} to obtain upper
  and lower bounds on $\Re(z^k)$ in terms of $\Re(z)^{k}$. 
  The bounds are sharpest when the ratio
  $\beta(z) \deq \Im(z)/\Re(z)$ has small magnitude.  
  Lemma~\ref{Lem:LemmaB} implies that,
  for $\lambda\in\B_\delta$,
  \begin{align*}
    |\beta(\psi(\lambda))| &\leq \frac{\sum_{i<j<k} |\lambda_{\{i,j\}}
      \lambda_{\{j,k\}} \lambda_{\{k,i\}}|
      +\epsilon_2(\lambda)}{1-\tfrac{1}{2}\|\lambda\|^2}\\
    &\leq \frac {\tfrac{1}{6}(n\delta)^3+\tfrac{1}{12}(n\delta)^4}
    {1-\tfrac{1}{4}(n\delta)^2} \leq \tfrac{1}{3}(n\delta)^3\,.
  \end{align*}
  The inequality \eqref{Eq:UBRez4t} then implies
  that
  \begin{equation} \label{Eq:Lower1}
    \Re(\psi(\gamma)^{4t})
    \leq
    \Re(\psi(\gamma))^{4t}
    \Bigl[ 1+ \frac{(n\delta)^6}{9} \Bigr]^{2t} \,.
  \end{equation}
  Proposition~\ref{RealPartOfPowerProposition} also supplies a lower
  bound for $\Re(\psi(\gamma)^{4t})$.  However, this bound only holds
  for powers $4t$ which satisfy the condition
  $\alpha(z) \deq 1-\binom{4t}{2}\beta^2(z)>0$. 

  We write $\beta$ and $\alpha$ for $\beta(\psi(\gamma))$ and
  $\alpha(\psi(\gamma)) = 1 - \binom{4t}{2}\beta^2$,
  respectively.    Suppose that $4t < 3(n\delta)^{-3}$.
  Then $4t \beta < 1$, and so $\alpha > 1/2$.
  In particular, $\alpha > 0$, whence \eqref{Ineq: z4t} can
  be applied to obtain
  \begin{align*}
    \Re( \psi(\gamma)^{4t} )
    & \geq \Re( \psi(\gamma))^{4t}\Big[1 + \beta^2
      \Big]^{2t}\Big[1 + \Big(\frac{4t}{\alpha}\Big)^2\beta^2
        \Big]^{-1/2} \\
    & \geq \Re(\psi(\gamma))^{4t}\Big[1 +
      \Big(\frac{\beta}{\alpha}\Big)^2 \beta^2 \Big]^{-1/2} \,.
  \end{align*}
  Since $\alpha \geq 1/2$ and $\beta^2 \leq
  (n\delta)^3/3$, it follows that $\beta^2/\alpha^2 \leq
  (4/9)(n\delta)^6$ and thus
  \begin{equation*}
    \Re(\psi(\gamma)^{4t})
      \geq \Re(\psi(\gamma))^{4t}
        \left[ 1 + \tfrac{4}{9}(4t)^2
        (n\delta)^6 \right]^{-\frac{1}{2}}\,.
  \end{equation*}
  The above bound and \eqref{Eq:Lower1} imply that
  \begin{equation} 
  \begin{split} \label{Eq:PreBounds}
    \left[1+\tfrac{4}{9}(4t)^2(n\delta)^6\right]^{-\frac{1}{2}}
    &\int_{\B_\delta}\Re\bigl(\psi(\gamma)\bigr)^{4t} d\gamma\\
    &\leq
    \int_{\B_\delta}\Re\bigl(\psi(\gamma)^{4t}\bigr) d\gamma\\
    &\leq \left[ 1+\tfrac{1}{9}(n\delta)^6 \right]^{2t}
    \int_{\B_\delta}\Re\bigl(\psi(\gamma)\bigr)^{4t} d\gamma
  \end{split}
  \end{equation}

  We now turn to bounding $\int_{\B_\delta} \Re(\psi(\gamma))^{4t}d\gamma$.
%
  Equation \eqref{RePsi}
  of Lemma~\ref{Lem:LemmaB} implies that
  \begin{equation*}
    \psi(\gamma)=e^{\frac{1}{2}\|\gamma\|^2}[1+\epsilon_1(\gamma)]\,,
  \end{equation*}
  where $|\epsilon_1(\gamma)|<\tfrac{1}{12}(n\delta)^4$.  Notice that
  this estimate is ideal when we need an estimate for powers of
  $\psi(\gamma)$.  Moreover, the real part of $1+\epsilon_1(\gamma)$
  must lie between $1+\tfrac{1}{12}(n\delta)^4$ and
  $1-\tfrac{1}{12}(n\delta)^4$.  Therefore, we have
  \begin{equation} \label{Eq:PreBounds2}
    \bigl[1+\tfrac{1}{12}(n\delta)^4\bigr]^{4t}e^{-\frac{4t}{2}\|\gamma\|^2}
    \geq
    \Re(\psi(\lambda))^{4t}
    \geq
    \bigl[1-\tfrac{1}{12}(n\delta)^4\bigr]^{4t}e^{-\frac{4t}{2}\|\gamma\|^2}
    \,.
  \end{equation}

  If we let
  \begin{equation*} 
    J(d,t,\delta)
    \deq \int_{\B_\delta}
    e^{-\frac{t}{2}\sum_{j<k}\gamma_{\{j,k\}}^2}\ d\gamma\,,
  \end{equation*}
  then \eqref{Eq:PreBounds} and \eqref{Eq:PreBounds2} 
  imply that
  \begin{equation} \begin{split} \label{Eq:PreBounds4}
    \Bigl[1 +  \tfrac{4}{9}(4t)^2 & (n\delta)^6\Bigr]^{-1/2}
    \left[1 - \tfrac{1}{12}(n\delta)^{4}\right]^{4t}
    J(d,4t,\delta) \\
    & \leq \int_{\B_\delta} \Re(\psi(\gamma))^{4t}
    d\gamma \\
    & \leq \left[1 + \tfrac{1}{9}(n\delta)^6\right]^{2t}
    \left[1 + \tfrac{1}{12}(n\delta)^4\right]^{4t}
    J(d,4t,\delta) \,.
  \end{split} \end{equation}

  To complete the proof, it remains to obtain suitable
  bounds for the integral $J(d,n,t)$.
  Changing variables by letting
  $\mu_{\{j,k\}}=\gamma_{\{j,k\}}\sqrt{t}$
  yields
  \begin{equation*}
    J(d,t,\delta)
    =t^{-d/2}\int_{\B_{\delta\sqrt{t}}}
    e^{-\frac{1}{2}\sum_{j<k}\mu_{\{j,k\}}^2}\ d\mu\,.
  \end{equation*}
  Now, for all $\rho>0$, we have
  \begin{equation*}
    \int_{0}^{\rho^2}2\pi r e^{-\frac{1}{2}r^2}\ d r
    <
    \int_{[-\rho,\;\rho\,]^2}
    e^{-\frac{1}{2}(\alpha^2+\beta^2)}\ d\alpha\; d\beta
    < 
    \int_{0}^{2\rho^2}2\pi r e^{-\frac{1}{2}r^2}\ d r\,.
  \end{equation*}
  So
  \begin{equation*}
    \sqrt{2\pi\left(1-e^{-\rho^2/2}\right)}
    <
    \int_{[-\rho,\,\rho]}e^{-\mu_{\{j,k\}}^2/2}d\mu_{\{j,k\}}
    <
    \sqrt{2\pi\left(1-e^{-\rho^2}\right)}\,.
  \end{equation*}
  Therefore
  \begin{multline*}
    \Big(\frac{2\pi}{t}\Big)^{d/2}\Big(1-e^{-t\delta^2/2}\Big)^{d/2}
    < J(d,t,\delta)
    <\Big(\frac{2\pi}{t}\Big)^{d/2}\Big(1-e^{-t\delta^2}\Big)^{d/2}\,.
  \end{multline*}
  Combining this with \eqref{Eq:PreBounds4},
  and using \eqref{Eq:RealInt},
  establishes \eqref{Eq:Verify}, completing the proof. 
\end{proof}

We can now derive the asymptotic formula for $N_{n,4t}$ 
in Theorem~\ref{Thm:AsymFormula}.
\begin{proof}[Proof of Theorem \ref{Thm:AsymFormula}]
  Fix $n$ and let $\delta = t^{-5/12}$.  Then
  $4t(n\delta)^3 = 4n^3 t^{-1/4}$, which for large enough $t$
  is less than $1$, so the bound \eqref{Eq:P00LB}
  can be used.  Note that $t^2 \delta^6 = t^{-1/2}$,
  whence $[1 + (4/9)(4t)^2(n\delta)^6]^{-1/2} \to 1$
  as $t \to \infty$.  Also, for any constant $c_n$, 
  \begin{equation*}
    \left( 1 - c_n t^{-5/3}\right)^t
    = e^{-c_n t^{-2/3}}[1 + o(1)]
    \to 1 \quad \text{as } t \to \infty \,,
  \end{equation*}
  whence $[1 - (1/12)(n\delta)^4]^{4t} \to 1$ as $t \to \infty$.
  Finally, $t\delta^2 = t^{1/6}$, so
  $[1-e^{-2t\delta^2}]^{d/2} \to 1$ as $t \to \infty$.
  We conclude that $L(n,4t,t^{-5/12}) \to 1$.
  This together with \eqref{Eq:P00LB} implies that
  \begin{multline} \label{Eq:LILB}
    \liminf_{t \to \infty}
    \frac{P_n^{4t}(0,0)}{2^{2d-n+1} (8\pi t)^{-d/2}} \\
    \geq \lim_{t \to \infty} \Big[L(n,4t,t^{-5/12}) 
      - \frac{e^{-\frac{11}{6}t^{1/6}}}{2^{2d-n+1} (8\pi t)^{-d/2}} \Big]
    = 1 \,.
  \end{multline}
  Similarly, $U(n,4t,t^{-5/12}) \to 1$, which with
  \eqref{Eq:P00UB} implies that
  \begin{multline} \label{Eq:LSUB}
    \limsup_{t \to \infty}
    \frac{P_n^{4t}(0,0)}{2^{2d-n+1} (8\pi t)^{-d/2}} \\
    \leq \lim_{t \to \infty} \Big[U(n,4t,t^{-5/12}) 
      - \frac{e^{-\frac{11}{6}t^{1/6}}}{2^{2d-n+1} (8\pi t)^{-d/2}} \Big]
    = 1 \,.
  \end{multline}
  The inequalities \eqref{Eq:LILB} and \eqref{Eq:LSUB},
  with the identity $N_{n,4t} = 2^{nt} P_n^{4t}(0,0)$, prove
 \eqref{Eq:AsymptoticFormula}.
\end{proof}

\section{Implications for the Existence and Abundance of Partial
  Hadamard Matrices} 
\label{Sec:ExistAndCount}
In this section, we show how our upper and lower
bounds for the integral $I(d,t)$ imply statements about the existence and
abundance of partial Hadamard matrices.  We will answer the following
questions:
\begin{itemize}
\item When does the upper bound \eqref{Eq:P00UB} imply a
  non-trivial bound on the number $N_{n,t}$ of partial Hadamard
  matrices?
\item When does the lower bound \eqref{Eq:P00LB} exceed zero, and
  hence imply that there is an $n\times t$ partial Hadamard matrix?
\end{itemize}

Let $N_{n,t}$ denote the number of $n\times t$ partial Hadamard
matrices, and let
\begin{align*}
  A(n,t)& \deq 2^{2d-n+1}(2\pi t)^{-d/2}\,,\\
  R(n,t)& \deq N_{n,t}/2^{nt}A(n,t)\,.
\end{align*}
Recalling the definitions in \eqref{Eq:Udef} and \eqref{Eq:Ldef},
let $U(n,t)$ and $L(n,t)$ be defined as 
\begin{align*}
  U(n,t) &\deq \min_{\delta<n^{-1}}\Bigl\{ U(n,t,\delta) +
  A(n,t)^{-1}e^{-\frac{11}{24}t\delta^2}
  \Bigr\}\,,\\
  L(n,t) &\deq \max_{\delta<n^{-1}}\Bigl\{ L(n,t,\delta) -
    A(n,t)^{-1}e^{-\frac{11}{24}t\delta^2} \Bigr\}\,.
\end{align*}
By Theorem~\ref{Thm:Bounds1},
\begin{equation*}
  L(n,4t)\ \leq\ R(n,4t)\ \leq\ U(n,4t) \,.
\end{equation*}

\begin{theorem} \label{Thm:NAbundExist}
  \hspace{0.1in}

  \begin{enumeratei}
  \item  \label{It:Abundance} {\bf Abundance of Designs}. 
    For all sufficiently large $n$, and $t>n^8$, 
    \begin{equation*}
      R(n,t) \leq e^{n^4 t^{-\frac{1}{2}}}
      +t^{\frac{d}{2}}e^{-\frac{11}{24}t^{\frac{1}{4}}} \,. 
    \end{equation*}
  \item \label{It:Existence} {\bf Existence of Designs}. 
    For all $\alpha,\beta>0$,
    and $n$ sufficiently large, we have
    \begin{equation*}
    L(n,t=n^{12+3\beta+2\alpha})
    > e^{-\frac{1}{4}n^{-2\alpha}}
    +A(n,t)^{-1}e^{-\frac{11}{24}n^{2+\beta}}\,.   
    \end{equation*}
    For all sufficiently large $n$, there is an $n\times 4t$ partial
    Hadamard matrix for all $t>n^{12}$.
  \end{enumeratei}
\end{theorem}

\begin{proof}
%
  We bound the function $U(n,t)$ by
  obtaining separate bounds for the logarithms of the two pieces:
  \begin{equation*}
    u_1(n,t,\delta) \deq A(n,t)^{-1}e^{-\frac{11}{24}t\delta^2}
    \quad \mbox{and}\quad
    U(n,t,\delta)\,.
  \end{equation*}
  For $U(n,t)$ to be small, the
  logarithm
  \begin{equation*}
     \log\left[ u_1(n,t,\delta) \right]=
    -\tfrac{11}{24}t\delta^2
    -(2d-n+1)\log2
    +\tfrac{d}{2}\log(2\pi)
    +\tfrac{d}{2}\log t
  \end{equation*}
  of the first piece must be negative and large in absolute value.  
  Therefore, for fixed $n$, as $t$
  grows the quantity $t\delta^2$ must grow.  So we put
  $\delta=t^{-\frac{1}{2}+\epsilon}$, where $\epsilon>0$ is small.
  Since we require $\delta<n^{-1}$, we must be sure that
  $t^{\frac{1}{2}-\epsilon}>n$.  In any case, setting
  $\delta=t^{-\frac{1}{2}+\epsilon}$, the expression for
  $\log(u_1(n,t,\delta))$ becomes
  \begin{equation*}
    f_{\epsilon}(t)=-\tfrac{11}{24}t^{2\epsilon}
    +\tfrac{d}{2}\log t
    -(2d-n+1)\log2
    +\tfrac{d}{2}\log(2\pi)\,.
  \end{equation*}
  Notice that
  \begin{equation*}
    -(2d-n+1)\log2 +\tfrac{d}{2}\log(2\pi)<0\,.
  \end{equation*}
  Therefore
  \begin{equation} \label{Eq:Sharp} 
    f_{\epsilon}(t)\leq-\tfrac{11}{24}t^{2\epsilon}
    +\tfrac{d}{2}\log t := g_\ep(t) \,.
  \end{equation}
  For $\epsilon>0$, the function $g_{\epsilon}(t)$ attains its maximum
  \begin{equation*}
    m(\epsilon)= \tfrac{1}{4}d\epsilon^{-1}  \left[ \log
    \left(\tfrac{6}{11}d\epsilon^{-1} \right) -1 \right]
  \end{equation*}
  at the point 
  $t_0(\epsilon)=\left(\tfrac{6}{11}d\epsilon^{-1} \right)^{1/2\epsilon}$.
  In particular,
  $t_0(1/8) = (\tfrac{48}{11}d)^4$
  and $m(1/8) = 2d\log\left(\tfrac{48}{11}de^{-1}\right)$.
  Note that $g_{1/8}(n^9)$ is approximately $-n^{9/4}$,
%
%
  so, once $t$ exceeds $n^8$, the function $f_\epsilon(t)$
  rapidly approaches zero.  In any case, putting $\epsilon=1/8$ in
  \eqref{Eq:Sharp}, yields the bound
  \begin{equation}\label{UpperFirstPartBound}
  u_1(n,t,t^{-\frac{3}{8}})
  \leq t^{\frac{d}{2}}e^{-\tfrac{11}{24}t^{\frac{1}{4}}}\,.
  \end{equation}

  We now examine the behavior of the second piece
  \begin{equation*}
    U(n,t,\delta) =\bigl[1+\tfrac{1}{9}(n\delta)^6\bigr]^{t/2}
  \bigl[1+\tfrac{1}{12}(n\delta)^4\bigr]^{t}
  \bigl[1-e^{-t\delta^2}\bigr]^{d/2} \,.
  \end{equation*}
  Since $1 + x \leq e^x$, we
  have
  \begin{equation*}
    \log U(n,t,\delta)
    \leq 
    \tfrac{1}{18} (n\delta)^6 t
    +\tfrac{1}{12} (n\delta)^4 t
    -\tfrac{1}{2} de^{-t\delta^2} \,.
  \end{equation*}
  Therefore,
  \begin{equation*}
    \log U(n,t,t^{-1/2+\epsilon})
    \leq 
    \tfrac{1}{18}n^6 t^{-2+6\epsilon}
    +\tfrac{1}{12}n^4 t^{-1+4\epsilon}
    -\tfrac{1}{2}de^{-t^{2\epsilon}}
    \,.
  \end{equation*}
  For $\epsilon<1/2$, and $n$ fixed, the middle term
  eventually dominates as $t$ grows.  For $\epsilon<1/4$,
  this term approaches zero.  Indeed, for $\epsilon=1/8$ and
  $t>n^8$, we have for $n$ sufficiently large
  \begin{equation*}
    U(n,t,t^{-\frac{3}{8}})
    \leq e^{n^4 t^{-\frac{1}{2}}}\,.
  \end{equation*}
  Combining this with \eqref{UpperFirstPartBound}, for $t>n^8$ and $n$
  sufficiently large, we have
  \begin{align*}
    U(n,t):=&\min_{\delta<n^{-1}}\{u_1(n,t,\delta)+U(n,t,\delta)\}\\
    &\leq u_1(n,t,t^{-\frac{3}{8}})+U(n,t,t^{-\frac{3}{8}})
    \leq e^{n^4 t^{-\frac{1}{2}}}
    +t^{\frac{d}{2}}e^{-\tfrac{11}{24}t^{\frac{1}{4}}}
    \,.
  \end{align*}
  This completes the proof of \eqref{It:Abundance}.
  
  We now prove \eqref{It:Existence}.
  Recall from \eqref{Eq:Ldef} that
  \begin{equation*}
  L(n,t,\delta)
   =\bigl[1+\tfrac{4}{9}t^2(n\delta)^6\bigr]^{-\frac{1}{2}}
    \bigl[1-\tfrac{1}{12}(n\delta)^4\bigr]^{t}
    \bigl[1-e^{-\frac{1}{2}t\delta^2}\bigr]^{d/2} \,.
  \end{equation*}
  For fixed $n$ and $t$, we determine when there
  exists $\delta\in[0,n^{-1}]$ such that
  \begin{equation}\label{logLLowerBound}
    \log L(n,t,\delta)
    > -\log A(n,t) -\tfrac{11}{24}t\delta^2\,,
  \end{equation}
  ensuring that $L(n,t)>0$. 
  Since $e^{-2x} \leq 1 - x$ for $0  \leq x \leq 1/2$, and, since $e^x>1+x$,
  \begin{equation*}
    \log L(n,t,\delta) \geq
    -\tfrac{2}{9} t^2 (n\delta)^6
    -\tfrac{1}{12}(n\delta)^4 t 
    - \tfrac{1}{4}d e^{-\frac{1}{2} t\delta^2} \,.
  \end{equation*}
  Thus \eqref{logLLowerBound} is satisfied if 
  the right-hand side above exceeds the right-hand side
  of \eqref{logLLowerBound}.  After rearranging, this
  is equivalent to
  \begin{multline*}
    t\bigl(\tfrac{11}{24}\delta^2 -\tfrac{1}{12}(n\delta)^4
    -\tfrac{2}{9}(n\delta)^6 t\bigr)\\
    \geq \tfrac{1}{2}d\log{t}+
    \tfrac{1}{2}d\bigl(\log(2\pi)+\tfrac{1}{2}e^{-\frac{1}{2}t\delta^2}\bigr)
    -(2d-n+1)\log(2)\,.
  \end{multline*}
  This inequality certainly holds if we drop the last two terms which
  contribute a comparatively small negative quantity as $t$ grows.
  Thus we are led to consider the simpler inequality
  \begin{equation} \label{Eq:Dagger} 
    t(\tfrac{11}{24}\delta^2
    -\tfrac{1}{12}(n\delta)^4
    -\tfrac{2}{9}(n\delta)^6 t)
    \geq
    \tfrac{1}{4}n^2\log t\,.
  \end{equation}
  This inequality presents three challenges which we need to
  overcome.  Firstly, we must ensure that the left-hand side is
  positive; so we must have
  \begin{equation} \label{Eq:Star} 
    \tfrac{11}{24}\delta^2
    \geq \tfrac{1}{12}(n\delta)^4
         +\tfrac{2}{9}(n\delta)^6 t\,.
  \end{equation}
  Secondly, we must find the smallest $t$ for a given $n$ for which
  the inequality has a feasible region for $\delta$.  Thirdly, we must
  ensure that the conditions 
  \begin{equation} \label{Eq:StarStar} 
    \delta<n^{-1}\qquad\mbox{and}\qquad t(n\delta)^3<1 \,,
  \end{equation}
  imposed by Theorem~\ref{Thm:Bounds1} hold.

  Before we begin, it is helpful to consider the following
  simplified version of \eqref{Eq:Dagger}:
  \begin{equation*}
    f_{a,b,c}(t)=\frac{t(a-bt)}{\log t}\geq c n^2\,,
  \end{equation*}
  where $a,b,c>0$.  For large $t$, the function $f_{a,b,c}(t)$ is
  essentially quadratic in $t$.  So ignoring the effect of the $\log
  t$ term, we should expect there to be a solution if $cn^2$ is less
  than the maximum $\tfrac{1}{4}a^2b^{-1}$ attained by the function
  $f_{a,b,c}$.  Moreover, the solution if it exists will lie in the
  interval $(0,a/b)$.  The corresponding maximum and interval for the
  inequality \eqref{Eq:Dagger} are (dropping the constant coefficients
  which are immaterial to this argument)
  \begin{equation*}
    (\delta^2(n\delta)^{-3}-(n\delta))^2\qquad\mbox{and}\qquad
    \Bigl(0,(n\delta)^{-2}(n^{-4}\delta^{-2}-1)\Bigr)\,.
  \end{equation*}
  Here we took $a=\delta^2-(n\delta)^4$ and $b=(n\delta)^6$.
  So, if there is a solution, we must have
  \begin{equation*}
    n<\delta^2(n\delta)^{-3}-(n\delta)<\delta^2(n\delta)^{-3}\,.
  \end{equation*}
  So, for $n$ large, we must have $\delta^{-1}>n^4$.  Putting
  $\delta=n^{-4-\epsilon}$, where $\epsilon>0$, in 
  \eqref{Eq:Dagger} yields
  the simplified inequality
  \begin{equation} \label{Eq:DaggerDagger} 
    \frac{t}{\log t}
    \geq
    n^{10+2\epsilon}
    \bigl(\tfrac{11}{6}
    -\tfrac{1}{3}n^{-4-2\epsilon}
    -\tfrac{8}{9}n^{-10-4\epsilon} t\bigr)^{-1}
    \,,
  \end{equation}
  where the roles of the various terms on the left-hand side of the
  original inequality \eqref{Eq:Dagger} are now clear.  In particular,
  we now see that for any $\epsilon>0$, as $n$ grows, there is a
  feasible region for $\delta$ when $t = n^{10+2\epsilon+\beta}$,
  where $\beta>0$, that the term $\tfrac{11}{24}t\delta^2$ is the
  important term, and that the term $\tfrac{2}{9}(n\delta)^6 t$
  presents no difficulty: i.e., the condition \eqref{Eq:Star} can be
  satisfied, provided that $\beta<2\epsilon$.  Indeed, since
  $\delta=n^{-4-\epsilon}$, the first part of condition
  \eqref{Eq:StarStar} is already satisfied.  However, the second part
  of \eqref{Eq:StarStar} requires that $t(n\delta)^3<1$, which holds
  if and only if $4n^{1-\epsilon+\beta}<3$.  Thus all conditions are
  satisfied for sufficiently large $n$ provided that $\beta>0$ and
  $\epsilon>1+\beta$.  Therefore, putting $\epsilon=1+\alpha+\beta$,
  we have
  \begin{equation*}
    t=n^{12+3\beta+2\alpha}
    \qquad \delta=n^{-5-\beta-\alpha}\qquad(\alpha,\beta>0)
  \end{equation*}
  and, for all sufficiently large $n$, these values for $t$ and $\delta$ 
  satisfy all conditions.  Feeding these parameters into the lower bound 
  \eqref{logLLowerBound} for $\log\bigl(L(n,t,\delta)\bigr)$ yields 
  the expression:
  \begin{equation*}
    -\tfrac{1}{12}(4n^{-4-\beta-2\alpha})
    -\tfrac{2}{9}n^{-2\alpha} 
    -\tfrac{1}{4}de^{-\frac{1}{2}n^{2+\beta}}\,.
  \end{equation*}
  The middle term dominates for large $n$; so for sufficiently large
  $n$,
  \begin{equation*}
    L(n,t=n^{12+3\beta+2\alpha},\delta=n^{-5-\beta-\alpha})
    >
    e^{-\frac{1}{4}n^{-2\alpha}}\,,
  \end{equation*}
  say, and, indeed,
  \begin{equation*}
    L(n,t=n^{12+3\beta+2\alpha})
    > e^{-\frac{1}{4}n^{-2\alpha}}
    +A(n,t)^{-1}e^{-\frac{11}{24}n^{2+\beta}}\,.   
  \end{equation*}
  This completes the proof of (ii).
\end{proof}

\section{The Branching Bound} \label{Sec:Branching}
In this section, we take advantage of the fact that the walk for
partial Hadamard matrices with $n$ rows contains, as projections, the
walks for the partial Hadamard matrices with fewer rows than $n$.  We
exploit this structure to obtain an upper bound on the number of
distinct $n\times 4t$ partial Hadamard matrices.  We call this the
{\em Branching Bound}.

The idea is that we can build up any $n\times t$ partial Hadamard
matrix by the searching a tree ${\mathcal T}$, say whose nodes at
level $m$ correspond to the $m\times t$ partial Hadamard matrices.
The parent of the node at level $m$ corresponding to the matrix $A$ is
the node at level $m-1$ corresponding to the partial Hadamard matrix
obtained by removing the last row of $A$.  If we choose a total order
on $\V_t$, then we can fully specify such a tree.  
Any total order on $\V_t$ imposes a total order on
the set of $m\times t$ partial Hadamard matrices:
matrix $A$ is greater than matrix $B$ if their first $j$ rows agree,
and the $(j+1)$-th row of matrix $A$ is greater than the $(j+1)$-th row
of $B$.  Then we may suppose the $i$-th node at level $m$ of 
${\mathcal T}$ corresponds to the $i$-th $m\times t$ partial Hadamard matrix.

The following lemma allows us to bound the number of nodes at level
$m+1$ in terms of the number of nodes at level $m$.

\begin{lemma} \label{Lem:Subspace} If $t \geq m$, then any
  $m$-dimensional real subspace of $\R^t$ contains at most $2^m$
  elements of $\V_t$.  Moreover, this bound can be attained for all
  $t\geq m$.
\end{lemma}
\begin{proof}
  Let $c^{(1)}, \ldots, c^{(\ell)}\in \V_t$ be $\ell$ vectors lying in
  some $m$-dimensional real subspace.  Form the $t\times \ell$ matrix
  $C$ whose $i$-th column is $c^{(i)}$:
  \begin{equation*}
    C =
    \begin{bmatrix}
      c^{(1)} & c^{(2)} & \ldots & c^{(\ell)}
    \end{bmatrix}\,.
  \end{equation*}
  If needs be, we can re-order the columns of $C$ so that the first
  $m$ columns of $C$ are linearly independent.  (If not, then there is
  no set of $m$ linearly independent columns, and the vectors all lie
  in an $m-1$ dimensional subspace.)  Moreover, since the $t \times m$
  matrix comprised of the first $m$ columns of $C$ has rank $m$, we
  may re-order the rows of $C$ so that the $m \times m$ matrix $B$ in
  the upper-left corner of $C$ is invertible.

  Now let $s > m$, and let $b^{(s)}$ be the $m$-dimensional vector
  comprised of the first $m$ components of $c^{(s)}$.  Since $c^{(s)}$
  is a linear combination of the vectors $c^{(1)}, \ldots, c^{(m)}$,
  and since $B$ is invertible, there is a unique $m$-dimension real
  vector $a^{(s)}$ such that $b^{(s)} = B a^{(s)}$.  Indeed,
  $c^{(s)}=Ca^{(s)}$.  Notice that if $c^{(u)}$ is a different column
  of $C$ such that $b^{(u)} = b^{(s)}$, then $c^{(u)}=c^{(s)}$.
  Therefore, since there are at most $2^m$ choices for $b^{(s)}$, we
  see that $\ell\leq2^m$.
\end{proof}

We can now prove the following theorem:

\begin{theorem}\label{Thm:BranchStepBound}
  $P_{n}^{(4t)}(0,0)\leq 2^{n-1-4t}P_{n-1}^{(4t)}(0,0)$.
\end{theorem}
\begin{proof}
  For $z \in \V_{n-1}$, define
  \begin{equation*}
    Q(z) = (z_j z_k)_{1 \leq j < k \leq n} \,.
  \end{equation*}
  The quantity $\one\{P\}$ equals $1$ if property $P$ holds,
  and zero otherwise.
  The number of $n \times t$ partial Hadamard matrices is 
  exactly
  \begin{equation*}
    \sum_{y^{(1)} \in \V_n} \cdots \sum_{y^{(t)} \in \V_n}
       \one\Big\{ \sum_{s=1}^t y^{(s)}_iy^{(s)}_j = 0
      \quad \forall\; 1 \leq i < j \leq n \Big\}  \,.
  \end{equation*}
  Letting $a_s = y^{(s)}_1$ for $s=1,\ldots, t$,
  this equals
  \begin{multline*}
    \sum_{\substack{a_1 \in \{-1,1\}\\
        (y^{(1)}_2,\ldots,y^{(1)}_n) \in \V_{n-1}}}
    \!\! \cdots \!\!
    \sum_{\substack{a_t \in \{-1,1\}\\
        (y^{(t)}_2,\ldots,y^{(t)}_n) \in \V_{n-1}}}
      \one\Big\{ 
        \begin{array}{l} \sum_{s=1}^t y^{(s)}_i y^{(s)}_j = 0 \\
      \forall\; 2 \leq i < j \leq n  \end{array} \Big\} \\
      \times \one\Big\{ \begin{array}{l} \sum_{s=1}^t a_s y^{(s)}_j = 0\\
      \forall\; 2 \leq j \leq n \end{array} \Big\} \,,
  \end{multline*}
  and the letting $z^{(s)} = (y^{(s)}_2, \ldots, y^{(s)}_n)
  \in \V_{n-1}$, we obtain
  \begin{multline} \label{Eq:Pn}
    P_{n}^{(t)}(0,0) \\
    = 2^{-tn} \sum_{z^{(1)},\cdots, z^{(t)} \in
      \V_{n-1}} \Bigoneb{ \sum_{j=1}^t Q(z^{(j)}) = 0 } \sum_{a \in
      \V_{t}} \Bigoneb{ \sum_{j=1}^t a_j z^{(j)} = 0 } \,.
  \end{multline}

  Next, we apply Lemma \ref{Lem:Subspace}.
  If $\{z^{(1)},\ldots,z^{(t)}\}$ is a set of $t$ elements of
  $\V_{n-1}$ satisfying $\sum_{j=1}^t Q(z^{(j)}) = 0$, then the $(n-1)
  \times t$ matrix $Z_t = [z^{(1)}, \ldots, z^{(t)}]$ has rank $n-1$.
  So the solutions to $Z_t a = 0$ thus are contained in a $t - (n-1)$
  vector subspace of $\R^t$.  Then, by Lemma \ref{Lem:Subspace}, the
  set of $a \in \V_t$ satisfying $\sum_{j=1}^t a_j z^{(j)} = 0$ has
  cardinality at most $2^{t-n+1}$.

  Thus,
  \begin{align*}
    P_{n}^{(t)}(0,0) & \leq 2^{-n+1}\cdot 2^{-t(n-1)} \sum_{z^{(1)},
      \cdots, z^{(t)} \in \V_{n-1}}
    \Bigoneb{ \sum_{j=1}^t Q(z^{(j)}) = 0 }\\
    & = 2^{-(n-1)}\,P_{n-1}^{(t)}(0,0)\,.
  \end{align*}
\end{proof}
The following corollary is immediate:
\begin{corollary} \label{Cor:BranchingBound}
  For $4t\geq n\geq s\geq1$,
  \begin{equation*}
    P_{n}^{(4t)}(0,0)\leq 2^{-\binom{n}{2}+\binom{s}{2}}\,P_{s}^{(4t)}(0,0)\,.
  \end{equation*}
  In particular there are at most $2^{\binom{n+1}{2}}$ Hadamard
  matrices of order $n$.
\end{corollary}
This bound is clearly inexact.  Direct arguments prove that
\begin{equation*}
  P_{2}^{(4t)}(0,0)=2^{-4t}\binom{4t}{2}\,,
  \qquad \text{and} \qquad
  P_{3}^{(4t)}(0,0)=2^{-8t}(4t)!/(t!)^4\,.
\end{equation*}
By using Stirling's Formula (with error bounds) to approximate
the binomial coefficients above, one can see that the asymptotic
formula in Theorem~\ref{Thm:AsymFormula} is actually
very good.

\section{Conclusion}
We have introduced a random walk for each integer $n\geq3$ in which
the probability of returning to the start of the walk after $t$ steps
is proportional to the number of distinct $n\times t$ partial Hadamard
matrices.  The behavior of this walk when $t$ is close to $n$ is of
particular interest.  This paper contains a preliminary analysis of
this walk using Fourier theory on the $d$-dimensional integer lattice
(here $d=\binom{n}{2}$) which shows how the walk behaves for $t$
polynomial in $n$.  Consequently, we are able to estimate the number
of distinct $n\times t$ partial Hadamard matrices for $t>n^{12}$.
Even this preliminary analysis yields new facts about designs.

This paper has also completed an important first step in the standard
Fourier-theoretic approach to walks in a discrete lattice.  We have
been able to give a fairly complete description of the set of points
$\lambda\in \B_\pi$, where the characteristic function
$\psi(\lambda)$ has magnitude equal to one.  In our case, the set has
interesting combinatorial structure: for example, each point in the
set corresponds to a graph on $n$ vertices all of whose degrees are
even.  We have also obtained some estimates for the characteristic
function by methods which give us a glimpse of the underlying
combinatorial questions which will need to be studied in order to
obtain better more global estimates for the characteristic function.

Finally, we note that the walks discussed in this paper are just one
example of a walk corresponding to a familiar kind of combinatorial
design.  For example, we have carried out elsewhere 
most of the steps in this paper for the walks corresponding to
balanced incomplete block designs. 

\appendix
\section{Some Inequalities}
In this appendix, we record and prove
inequalities which relate $\Re(z)^t$ and
$\Re(z^t)$.

We employ the following version of the Neyman-Pearson Lemma:
\begin{lemma}\label{SumsOfRatiosLemma}
  Let $\lambda_0,\lambda_1,\dots,\lambda_n$ be positive real numbers,
  and let 
\[
A_0,A_1,\dots, A_n
\qquad\mbox{and}\qquad 
B_0,B_1,\dots, B_n
\] 
be non-negative real numbers.  Then
  \begin{equation*}
    \min_{0 \leq s \leq n} \Bigl\{\frac{B_s}{A_s}\Bigr\}
    \leq\frac{\sum_{s=0}^n\lambda_sB_s}{\sum_{s=0}^n\lambda_sA_s}
    \leq
    \max_{0 \leq s \leq n}\Bigl\{\frac{B_s}{A_s}\Bigr\}\,.
  \end{equation*}
\end{lemma}
\begin{proof}
  Let $s_0$ and $s_1$ satisfy
  \begin{equation*}
    \frac{B_{s_0}}{A_{s_0}}
    =\min_{0 \leq s \leq n}\Bigl\{\frac{B_s}{A_s}\Bigr\}
    \qquad\mbox{and}\qquad
    \frac{B_{s_1}}{A_{s_1}}
    =\max_{0 \leq s \leq n}\Bigl\{\frac{B_s}{A_s}\Bigr\}\,.
  \end{equation*}
  Then
  \begin{multline*}
    \frac{B_{s_0}}{A_{s_0}}
    = \frac{\sum_{0 \leq s \leq n}\lambda_sA_s(B_{s_0}/A_{s_0})}{\sum_{0 \leq s \leq n}\lambda_sA_s}
    \leq\frac{\sum_{0 \leq s \leq n}\lambda_sB_s}{\sum_{0 \leq s \leq
        n}\lambda_sA_s} \\
    \leq\frac{\sum_{0 \leq s \leq n}\lambda_sA_s(B_{s_1}/A_{s_1})}{\sum_{0 \leq s \leq n}\lambda_sA_s}
    = \frac{B_{s_1}}{A_{s_1}}\,.
  \end{multline*}
\end{proof}
We can now prove the required inequalities relating $\Re(z^t)$ to
$\Re(z)^t$.
\begin{proposition}\label{RealPartOfPowerProposition}
  Let $t>0$ be an integer.
  \begin{enumeratei}
  \item\label{RealPartOfPowerLemmaPart1} For any complex number we
    have
    \begin{equation}\label{Eq:z4t2}
      \Bigl\{\Re(z^{4t})
      \Bigl(
      1+\Bigl[\frac{\Im(z^{4t})}{\Re(z^{4t})}\Bigr]^2
      \Bigr)^{\frac{1}{2}}\Bigr\}^2
      =
      \Bigl\{
      \Bigl(
      1+\Bigl[\frac{\Im(z)}{\Re(z)}\Bigr]^2
      \Bigr)^{2t}
      \Re(z)^{4t}\Bigr\}^2\,.
    \end{equation}
    In particular,
    \begin{equation}\label{Eq:UBRez4t}
      \Re(z^{4t})
      \leq
      \Re(z)^{4t}
      \Bigl(
      1+\Bigl[\frac{\Im(z)}{\Re(z)}\Bigr]^2
      \Bigr)^{2t}\,.
    \end{equation}
  \item\label{RealPartOfPowerLemmaPart2} If
    $\alpha=1-\binom{4t}{2}\Bigl\{\frac{\Im(z)}{\Re(z)}\Bigr\}^{2}>0$,
    then
    \begin{equation}\label{Eq:z4t}
      \Re(z^{4t})
      =
      \Re(z)^{4t}
      \Bigl(
      1+\Bigl[\frac{\Im(z^{4t})}{\Re(z^{4t})}\Bigr]^2
      \Bigr)^{-\frac{1}{2}}
      \Bigl(1+\Bigl[\frac{\Im(z)}{\Re(z)}\Bigr]^2\Bigr)^{2t}
      \,.
    \end{equation}
  \item\label{RealPartOfPowerLemmaPart3} If
    $\alpha=1-\binom{4t}{2}\Bigl\{\frac{\Im(z)}{\Re(z)}\Bigr\}^{2}>0$,
    then
    \begin{equation}\label{Eq:ImzOverRez}
      \Bigl[\frac{\Im(z^{4t})}{\Re(z^{4t})}\Bigr]^2
      \leq
      \Bigl[\frac{4t}{\alpha}\Bigr]^2
      \Bigl[\frac{\Im(z)}{\Re(z)}\Bigr]^2\,.
    \end{equation}
  \item\label{RealPartOfPowerLemmaPart4} If
    $\alpha=1-\binom{4t}{2}\Bigl\{\frac{\Im(z)}{\Re(z)}\Bigr\}^{2}>0$,
    then
    \begin{equation}\label{Ineq: z4t}
      \Re(z^{4t})
      \geq
      \Re(z)^{4t}
      \Bigl(1+\Bigl[\frac{\Im(z)}{\Re(z)}\Bigr]^2\Bigr)^{2t}
      \Bigl(
      1+
      \Bigl[\frac{4t}{\alpha}\Bigr]^2
      \Bigl[\frac{\Im(z)}{\Re(z)}\Bigr]^2
      \Bigr)^{-\frac{1}{2}}
      \,.
    \end{equation}
  \end{enumeratei}
\end{proposition}
\begin{proof}
  For any complex number $z$ and any natural number $t$, we have
  \begin{align*}
    \Bigl[\Re(z^{4t}) \Bigl\{
    1+\Bigl(\frac{\Im(z^{4t})}{\Re(z^{4t})}\Bigr)^2
    \Bigr\}^{\frac{1}{2}}\Bigr]^2 &=\Re(z^{4t})^2 \Bigl\{
    1+\Bigl(\frac{\Im(z^{4t})}{\Re(z^{4t})}\Bigr)^2 \Bigr\}
    =|z^{4t}|^2\,,
  \end{align*}
  and
  \begin{align*}
    (|z|^2)^{4t} &= \Bigl[ \Bigl\{
    1+\Bigl(\frac{\Im(z)}{\Re(z)}\Bigr)^2 \Bigr\}\Re(z)^2\Bigr]^{4t} =
    \Bigl[ \Bigl\{ 1+\Bigl(\frac{\Im(z)}{\Re(z)}\Bigr)^2 \Bigr\}^{2t}
    \Re(z)^{4t}\Bigr]^2\,.
  \end{align*}
  So equation \eqref{Eq:z4t2} holds.  This proves
  part~\eqref{RealPartOfPowerLemmaPart1}.

  To prove part~\eqref{RealPartOfPowerLemmaPart2}, we must show that if
  $\alpha>0$, then $\Re(z^{4t})$ is non-negative.  Suppose $z=a+\rmi
  b$ where $a$ and $b$ are real.  Then
  \begin{equation*}
    \Re(z^{4t})=b^{4t}+a^{4t}
    \sum_{s=0}^{t-1}\binom{4t}{4s}\Bigl(\frac{b}{a}\Bigr)^{4s}
    \Bigl[
    1-\frac{(4t-4s)(4t-4s-1)}{(4s+1)(4s+2)}\Bigl(\frac{b}{a}\Bigr)^2
    \Bigr]\,,
  \end{equation*}
  and
  \begin{align*}
    \Im(z^{4t}) & =a^{4t}
    \sum_{s=0}^{t-1} \Biggl\{ \binom{4t}{4s+1}
    \Bigl(\frac{b}{a}\Bigr)^{4s+1} \\
    & \qquad \qquad \times
    \Bigl[
    1-\frac{(4t-4s-1)(4t-4s-2)}{(4s+2)(4s+3)}\Bigl(\frac{b}{a}\Bigr)^2
    \Bigr] \Biggr\} \,.
  \end{align*}
  Since, for $s\in\{0,1,\dots,t-1\}$,
  \[ 
  \Bigl[
  1-\frac{(4t-4s)(4t-4s-1)}{(4s+1)(4s+2)}\Bigl(\frac{b}{a}\Bigr)^2
  \Bigr]
  \geq  1-\binom{4t}{2}\Bigl\{\frac{b}{a}\Bigr\}^{2}
  = \alpha>0\,,
  \]
  we have $\Re(z^{4t})>0$.  This proves
  part~\eqref{RealPartOfPowerLemmaPart2}.

  We prove part~\eqref{RealPartOfPowerLemmaPart3}.  If we put
  \begin{equation*}
    \lambda_0=1,\quad\lambda_s=\binom{4t}{4s}\Bigl(\frac{b}{a}\Bigr)^{4s}\,,
  \end{equation*}
  \begin{equation*}
    A_0=b^{4t}\,,\quad A_s=
    \Bigl[
    1-\frac{(4t-4s)(4t-4s-1)}{(4s+1)(4s+2)}\Bigl(\frac{b}{a}\Bigr)^2
    \Bigr]\,,
  \end{equation*}
  and
  \begin{equation*}
    B_0=0\,,\quad B_s=
    \Bigl(\frac{4t-4s}{4s+1}\Bigr)
    \Bigl[
    1-\frac{(4t-4s-1)(4t-4s-2)}{(4s+2)(4s+3)}\Bigl(\frac{b}{a}\Bigr)^2
    \Bigr]\,,
  \end{equation*}
  then
  \begin{equation*}
    \Re(z^{4t})=\sum_{s=0}^{t-1}\lambda_sA_s\,,
    \qquad\mbox{and}\qquad
    \Im(z^{4t})=
    \Bigl(\frac{b}{a}\Bigr)
    \sum_{s=0}^{t-1}\lambda_sB_s\,,
  \end{equation*}
  and
  \begin{align*}
    \max_{s}\Bigl\{\frac{B_s}{A_s}\Bigr\} &= \max_{s} \Bigl\{
    \Bigl(\frac{4t-4s}{4s+1}\Bigr) \frac { \Bigl[
      1-\frac{(4t-4s-1)(4t-4s-2)}{(4s+2)(4s+3)}\Bigl(\frac{b}{a}\Bigr)^2
      \Bigr] } { \Bigl[
      1-\frac{(4t-4s)(4t-4s-1)}{(4s+1)(4s+2)}\Bigl(\frac{b}{a}\Bigr)^2
      \Bigr] } \Bigr\}\,.
  \end{align*}
  The argument of the right-hand side is maximized when $s=0$.  So
  \begin{align*}
    \max_{s}\Bigl\{\frac{B_s}{A_s}\Bigr\} &= \frac{4t}{\alpha} \Bigl[
    1-\frac{1}{3}\binom{4t-1}{2}\Bigl(\frac{b}{a}\Bigr)^2 \Bigr] \leq
    \frac{4t}{\alpha} \,.
  \end{align*}
  Therefore, applying Lemma~\ref{SumsOfRatiosLemma}, we have
  \begin{equation*}
    \Bigl[
    \frac{\Im(z^{4t})}{\Re(z^{4t})}
    \Bigr]^2
    \leq
    \Bigl[\frac{4t}{\alpha}\Bigr]^2
    \Bigl[
    \frac{\Im(z)}{\Re(z)}
    \Bigr]^2\,.
  \end{equation*}
  This proves part~\eqref{RealPartOfPowerLemmaPart3}.  Finally,
  substituting \eqref{Eq:ImzOverRez} into \eqref{Eq:z4t} gives
  part~\eqref{RealPartOfPowerLemmaPart4}.
\end{proof}


\begin{thebibliography}{1}

\bibitem{billingsley}
P.~Billingsley.
\newblock {\em Probability and measure}.
\newblock Wiley Series in Probability and Mathematical Statistics. John Wiley
  \& Sons Inc., New York, third edition, 1995.
\newblock A Wiley-Interscience Publication.

\bibitem{Crai}
R.~Craigen.
\newblock Signed groups, sequences, and the asymptotic existence of {H}adamard
  matrices.
\newblock {\em J. Combin. Theory Ser. A}, 71(2):241--254, 1995.

\bibitem{dLG:CHC}
W.~de~Launey and D.~M. Gordon.
\newblock A comment on the {H}adamard conjecture.
\newblock {\em J. Combin. Theory Ser. A}, 95(1):180--184, 2001.

\bibitem{GS:RSA}
S.~W. Graham and I.~E. Shparlinski.
\newblock On {RSA} moduli with almost half of the bits prescribed.
\newblock {\em Discrete Appl. Math.}, 156(16):3150--3154, 2008.

\bibitem{spitzer}
F.~Spitzer.
\newblock {\em Principles of random walks}.
\newblock Springer-Verlag, New York, second edition, 1976.
\newblock Graduate Texts in Mathematics, Vol. 34.

\end{thebibliography}
\end{document}